\documentclass[a4paper,reqno,11pt]{amsart}
\usepackage{mathabx}
\usepackage{lmodern}
\usepackage[utf8]{inputenc}
\usepackage{amsfonts}
\usepackage{amsmath}
\usepackage{amssymb}
\usepackage{float}
\usepackage{tikz-cd}
\usepackage{graphicx}
\usepackage{bbm}
\usepackage{amscd}
\usepackage[sort&compress,numbers]{natbib}
\usepackage[left=2cm,right=2cm,bottom=3cm,top=3cm]{geometry}
\usepackage{etoolbox}
\usepackage{filecontents}
\usepackage{enumerate}
\usepackage[hidelinks]{hyperref}
\usepackage{parskip}
\usepackage{xspace} 
\usepackage{comment}

\usepackage[normalem]{ulem}

\usepackage{tikz}
\usetikzlibrary{matrix,arrows}

\theoremstyle{plain}
\newtheorem{theorem}{Theorem}[section]
\newtheorem*{theorem*}{Theorem}
\newtheorem{proposition}[theorem]{Proposition}
\newtheorem{conjecture}[theorem]{Conjecture}
\newtheorem{corollary}[theorem]{Corollary}

\newtheorem{example}[theorem]{Example}

\newtheorem{lemma}[theorem]{Lemma}

\newtheorem{definition}[theorem]{Definition}

\theoremstyle{remark}

\newtheorem{remark}[theorem]{Remark}
\newtheorem{claim}{Claim}[theorem]

\newtheorem{ejemplo}{{\sc Example}}
\newtheorem{notas}[theorem]{{\sc Remark}}
\newcommand{\nrm}[1]{\|#1\|}

\DeclareMathOperator{\fr}{\mathrm{Fr}}
\DeclareMathOperator{\stab}{\mathrm{Stab}}
\newcommand{\auh}{$\mathrm{(AuH)}$\xspace}
\DeclareMathOperator{\env}{\mathrm{Env}}

\newcommand{\prop}{\begin{proposition}}
\newcommand{\fprop}{\end{proposition}}
\newcommand{\cor}{\begin{corollary}}
\newcommand{\fcor}{\end{corollary}}

\newcommand{\defi}{\begin{definition}\rm}
\newcommand{\fdefi}{\end{definition}}
\newcommand{\eje}{\begin{ejemplo}}
\newcommand{\feje}{\end{ejemplo}}
\newcommand{\ejes}{\begin{ejemplos}}
\newcommand{\fejes}{\end{ejemplos}}
\newcommand{\lema}{\begin{lemma}}
\newcommand{\flema}{\end{lemma}}
\newcommand{\teor}{\begin{theorem}}
\newcommand{\fteor}{\end{theorem}}
\newcommand{\nota}{\begin{notas}\rm}
\newcommand{\fnota}{ \end{notas}}
\newcommand{\clam}{\begin{claim}}
\newcommand{\fclam}{\end{claim}}
\newcommand{\clams}{\begin{claim*}}
\newcommand{\fclams}{\end{claim*}}

\newcommand{\lclam}{\begin{lclaim}}
\newcommand{\flclam}{\end{lclaim}}
\newcommand{\prucl}{\prue[Proof of Claim:]}
\newcommand{\fprucl}{\fprue}
\newcommand{\ben}{\begin{enumerate}}
\newcommand{\een}{\end{enumerate}}
\newcommand{\bit}{\begin{itemize}}
\newcommand{\falta}{\textbf{******  falta **** }}

\newcommand{\eit}{\end{itemize}}

\newcommand{\mc}[1]{\mathcal{#1}}

\newcommand{\mr}[1]{\mathrm{#1}}

\newcommand{\casos}{\begin{itemize}}
\newcommand{\fcasos}{\end{itemize}\setcounter{cs}{1}}

\newcommand{\Gurarij}{{$\mathbb{G}$\ }}

\newcommand{\conj}[2]{ \{ {#1}\,:\,{#2} \} }

\newcommand{\buit}{\emptyset}

\newcommand{\ga}{\gamma}
\DeclareMathOperator{\id}{\mathrm{Id}}

\newcommand{\Ga}{\Gamma}

\newcommand{\de}{\delta}
\newcommand{\De}{\Delta}
\newcommand{\la}{\lambda}

\newcommand{\sig}{\sigma}
\newcommand{\Sig}{\Sigma}
\newcommand{\vphi}{\varphi}

\DeclareMathOperator{\isomo}{\mathrm{GL}}
\DeclareMathOperator{\isome}{\mathrm{Isom}}

\DeclareMathOperator{\Emb}{\mathrm{Emb}}

\newcommand{\vep}{\varepsilon}

\newcommand{\R}{{\mathbb R}}

\newcommand{\N}{{\mathbb N}}

\newcommand{\rest}{\upharpoonright}

\newcommand{\supp}{\mathrm{supp\, }}

\newcommand{\lat}{{\rm Lat}}

\newcommand{\con}{\subseteq}

\newcommand{\prue}{\begin{proof}}
\newcommand{\fprue}{\end{proof}}


\thanks{\textcolor{black}{ V.F. \ was supported by FAPESP, project 2016/25574-8, and by CNPq, grant 303731/2019-2. J.L.-A. \ was
partially supported by  the Ministerio de Ciencia e Innovación grant PID2019-107701GB-I00 (Spain) and  Fapesp grant
2016/25574-8 (Brazil) }}

\begin{document} 

\title{Envelopes in Banach spaces}
\author{V. Ferenczi, J. Lopez-Abad}

\maketitle

\
\begin{center}
\today    
\end{center}
\begin{abstract} 
We define the notion of isometric envelope of a subspace in a Banach space, and relate it  to a) the mean ergodic projection on the space of fixed points of a semigroup of contractions,  b)  results on Korovkin sets from the 70's, and c) extension properties of linear isometric embeddings.
We use this concept to address the recent conjecture 
that the Gurarij space and the spaces $L_p$, $p \notin 2\N+4$ are the only separable Approximately Ultrahomogeneous  
Banach spaces (a certain  multidimensional
transitivity of the action of the linear isometry group on the space considered in \cite{FLMT}). The similar conjecture for  Fra\"iss\'e Banach spaces (a strenghtening of the Approximately Homogeneous Property) is also considered. We characterize the Hilbert space as the only separable reflexive space in which any closed subspace coincides with its envelope.
We compute some envelopes in the case of Lebesgue spaces, showing that the reflexive $L_p$-spaces are the only reflexive rearrangement invariant spaces on $[0,1]$ for which all $1$-complemented subspaces are envelopes. We also identify the isometrically unique "full" quotient space of $L_p$ by a Hilbertian subspace, for appropriate values of $p$, as well as the associated topological group embedding of the unitary group into the isometry group of $L_p$.
\end{abstract}

\


\section{Introduction and notation}

\subsection{The motivation}

An initial motivation for this work stems from the the famous rotations problem of Mazur. 
Assume the linear isometry group of a separable Banach space $X$ acts transitively on its unit sphere; must $X$ be isometric to the Hilbert space? (see for example \cite{CabelloSanchez} or the recent survey \cite{felixvalentinbeata}). One may find many examples of non-Hilbertian spaces which are almost transitive (AT), meaning that the orbits of the action of the isometry group on the sphere are dense. Striking examples of this are the Lebesgue spaces $L_p(0,1)$ for $1 \leq p<\infty$. For $p \neq 2$ those spaces, although not transitive,  are very close to being so. Indeed they admit exactly two orbits for the action of the isometry group, the orbit of functions with full support, resp. without full support. In some sense this suggests that the $L_p$ spaces are transitive ``with respect to the supports", that is, provided one sees each function $f$ inside its ``appropriate $L_p$-space" $L_p({\rm supp\ }f)$.
The concept of isometric envelope defined in the present paper initially aimed to imitate some aspects of the support of a function, in an abstract  setting which would not require the structure of a function space.

The correct setting for defining the isometric envelope is multidimensional, so let us recall a few facts. A Banach space $X$ is ultrahomogeneous when any isometric map between finite dimensional subspaces extends to a (surjective) isometry on $X$. A multidimensional version of Mazur's rotation problem asks: must any separable ultrahomogeneous Banach space   be isometric to the Hilbert space (\cite{FLMT}).

The separable Lebesgue spaces are also very close to being ultrahomogeneous, although one has to exclude the even values of $p\ge 4$. For $p \notin 2\N+4$, $L_p(0,1)$ is approximately ultrahomogeneous \auh, meaning that any isometric map between finite dimensional subspaces may be approximated as accurately as wished by a (surjective) isometry (essentially due to Lusky \cite{Lusky}). The Gurarij space is also \auh \cite{KubisSolecki}. A conjecture formulated in \cite{FLMT} and which is our main interest in this paper is the following: is every separable \auh space isometric either to some $L_p(0,1)$ or to the Gurarij space?

Again isometric maps between subspaces of the \auh $L_p$ extend to isometries between respective appropriate $L_p$-subspaces ``enveloping" them, a consequence of Plotkin and Rudin equimeasurability theorem \cite{Plotkin, Rudin}, see also \cite{koldo} (the ``envelope" terminology is ours and does not appear in those works).  
Based on these examples, our aim is to define envelopes of subspaces in general Banach spaces, so that \auh spaces  can be proved to be ``ultrahomogeneous with respect to envelopes", in a precise sense to be defined later. The notion of envelope is then expected to be   useful to address the conjecture from \cite{FLMT}.

Let us note here another motivation for studying
\auh spaces. These are natural universal objects in the sense that two separable \auh are isometric as soon as they have the same finite dimensional subspaces up to isometry \cite[Proposition 2.22]{FLMT}. In fact, for a strengthening of the \auh property, called the {\em Fraïssé} property (see Definition \ref{fraisse}), there is a correspondence between these spaces and classes   of finite dimensional spaces with a natural amalgamation property, and, in this way, the separable Fraïssé Banach space corresponding to one of such classes $\mc F$ (called the Fraïssé limit of $\mc F$) is the ``generic'' universal space for all those spaces having dense subspaces consisting of an increasing union of elements in $\mc F$. Examples of this situation are precisely the Gurarij space and the spaces $L_p[0,1]$, for $p\notin 2 \N+4$, that are the generic universal spaces corresponding to the amalgamation classes of all separable   spaces, and the  class of $L_p$-spaces, respectively (see \cite{FLMT}). They are the only known examples, and it is unknown whether there exist \auh spaces which are not Fra\"iss\'e. Moreover, the Kechris-Pestov-Todorcevic correspondence for \auh Banach spaces characterizes  a fixed point property, called extreme amenability,  of their linear isometry group in terms of 
an approximate Ramsey property, of the collection of their finite dimensional subspaces. All the known \auh Banach spaces have this property (see \cite{Gr-Mi} for  $\ell_2$, \cite{Gi-Pe} and \cite{FLMT} for the other $L_p$'s, and \cite{BLMT} for the Gurarij space).     New examples of \auh Banach spaces could provide new extremely amenable groups; on the other hand if all separable \auh Banach spaces are isometric the $L_p$'s or the Gurarij, then, in some sense, their isometry groups appear as canonical extremely amenable groups of isometries of Banach spaces. 
Let us note finally that Fra\"iss\'e spaces must contain an isometric copy of the Hilbert space, \cite[Proposition 2.13.]{FLMT}, and therefore the study of the envelopes of hilbertian subspaces of  Fra\"iss\'e spaces will be particularly relevant.

Here is the definition of the isometric envelope $\env(Y)$ of a subspace $Y$ of $X$: it is the subspace of $x \in X$ such that $(T_i(x))_i$ converges to $x$ whenever the net $(T_i)_i$ in $\isome(X)$ converges pointwise on $Y$ to $\id_Y$. Under reflexivity we may  relate the isometric envelope to ergodic decompositions relative to semigroups of contractions, in the so-called Jacobs-de Leeuw- Glicksberg theory, as well as to results on Korovkin sets (see \cite{Wulbert}). In particular, we prove the fundamental result that if $X$ is reflexive strictly convex, then $\env(Y)$ is always $1$-complemented, and if $X$ is reflexive locally uniformly convex, then $\env(Y)$ is actually the smallest superspace of $Y$ complemented by a projection in the ``isometric hull" $\overline{\rm conv}(\isome(X))$. As a byproduct, the Hilbert space is the only separable reflexive space on which all subspaces are equal to their envelopes. We conclude Section 2 with a fundamental extension property of the envelope, which we state here for \auh spaces $X$: any partial isometry between subspaces $Y, Z$ of $X$ extends uniquely to a partial isometry between their envelopes $\env(Y)$ and $\env(Z)$; if $Y$ is separable then this extension map is SOT-SOT continuous and $\env(Y)$ is maximum among superspaces of $Y$ satisfying this property. Of particular interest are the ``full" subspaces of $X$, i.e. those whose envelope is equal to $X$: any partial isometry between full subspaces extends uniquely to a surjective isometry on $X$, and the isometry group of a full subspace embeds naturally into ${\rm Isom}(X)$. Thus we achieve the extension properties that envelopes were designed to satisfy.

In Section 3 we use the previous characterization to prove an important technical result: if $X$ is reflexive and separable, then the envelope $\env(Y)$ of the closure $Y$ of a directed sequence $\bigcup_i Y_i$ of subspaces coincides with the closure of the directed sequence $\bigcup_i \env(Y_i)$ of their envelopes.
We then investigate envelopes in $L_p$ spaces, and use the above to prove, for example, that if $Y$ is unital, the envelope coincides with some classical notions such as 
the sublattice generated by $Y$, the minimal $1$-complemented subspace containing $Y$, or the range of the conditional expectation associated to $Y$.
In particular, all $1$-complemented unital subspaces are envelopes in those spaces, and we prove that the $L_p$'s, $1<p<\infty$ are the only r.i. reflexive spaces on $[0,1]$ for which this happens.

After proving that, for $p \notin 2\N+4$, 
the  $L_p$'s contain a full copy of the Hilbert space, we apply the isometric extension results 
of Section 2 to study the associated exact sequence of Banach spaces and define objects such as the group embedding of $U(\ell_2)$ inside $\isome(L_p)$ or the associated ``full quotient" of $L_p$ by $\ell_2$. 

In the final section, we establish relations between the AUH or Fra\"iss\'e property and geometrical conditions  from local theory such as type or cotype, or complementations of euclidean or hilbertian subspaces. With these estimates we obtain ``local versions" of the known fact that $L_p$-spaces are Fra\"iss\'e if and only if $p \notin 2\N+4$. 
With this we reinforce the conjecture that all separable reflexive Fra\"iss\'e spaces must be $L_p$-spaces as well as identify natural intermediary steps towards proving this fact. We also answer a question of G. Godefroy by proving that any Fraïssé Banach space with a $C_\infty$-norm must be isomorphic to a Hilbert space.

We shall define envelopes in arbitrary Banach spaces, but with a focus on the reflexive case. We aimed at identifying the minimal hypotheses regarding the norm for each of our results (strict convexity/local uniform convexity, strict convexity of the dual, etc...). 
However it should be noted that \auh spaces are always
 uniformly convex and uniformly smooth as soon as they are reflexive. Therefore the reader mainly interested in the concept of envelope inside a reflexive \auh space $X$ can safely assume that the norm on $X$ is always uniformly convex and uniformly smooth.

\subsection{Notation}


Given a Banach space $X=(X,\nrm{\cdot})$ we shall write ${\rm Sph(X)}$ to denote the unit sphere of $X$ and $B_X$ to denote its unit ball. We write $\isomo(X)$ and $\isome(X)$ to denote the group of surjective linear isomorphisms and surjective linear isometries, respectively, ${\mathcal L}_1(X)$ to denote the semigroup of contractions on $X$, and if $X$ is a Banach lattice,  ${\mathcal L}_1^+(X)$ to denote the semigroup of positive contractions.

For the definition of classical properties of norms (strict convexity, uniform convexity, etc...) and related notions we refer to \cite{DGZ} and/or \cite{LT}. For completeness we recall the maybe slightly less classical notion of local uniform convexity. A Banach space $(X,\|.\|)$ is locally uniformly convex if $\forall x_0 \in {\rm Sph(X)}$ and $\forall \varepsilon>0$, there exists $\delta>0$ such that whenever $x \in B_X$,
$$\|\frac{x+x_0}{2}\|>1-\delta \Rightarrow \|x-x_0\|<\varepsilon.$$
Locally uniformly convex norms are also called locally uniformly rotund and for these reasons this property is abbreviated as ``LUR".

If $T$ is an operator on $X$ we denote by
${\rm Fix}(T)$ the closed subspace of fixed points of $T$, i.e. ${\rm Fix}(T)=\{x \in X: Tx=x\}$. If $S$ is a semigroup of operators on $X$ then
${\rm Fix}(S)$ denotes the closed subspace of fixed points of $S$, i.e.
${\rm Fix}(S):=
\bigcap_{T \in S}{\rm Fix}(T)$ is the largest subspace where all $T \in S$ act as the identity.
On the other hand, if $Y$ is a subspace of $X$, then  
${\rm Stab}_S(Y)$ denotes the semigroup of operators in $S$ acting as the identity on $Y$.


\subsubsection*{Acknowledgements}

We are grateful to  G. Godefroy and  W. T. Johnson for
helpful conversations and remarks.

\section{Envelopes}

An (abstract) envelope is a  set-valued map $e$ to every subset $A$ of a Banach space $X$ its envelope, i.e. a closed subspace $e(A)$ of $X$ satisfying the following properties
\begin{enumerate}[(a)]
\item $A\subseteq e(A)$.
\item If  $B \subseteq A$, then $ e(B) \subseteq e(A)$.
\item $e(\overline{\rm span}(A))=e(A)$.
\item $e(e(A))=e(A )$.
\end{enumerate}
So, $e$ is a standard (algebraic) closure operator with the additional property {\rm(c)}. 
Alternatively, any envelope map on $X$ may be fully described by the class ${\mathcal E}$ of closed subspaces of $X$ which are envelopes, and by the condition that $e(A)$ is the smallest element of ${\mathcal E}$ containing $A$ as a subset.
It is routine to check that a necessary and sufficient condition for a given class ${\mathcal E}$ to define an envelope map in this manner is to be stable by (finite or infinite) intersections.

Perhaps the more geometrical example of envelope is the {\em minimal} envelope that assigns to a subset, if it exists, the smallest superspace of it that is 1-complemented. This set-mapping is not always well defined: there are  examples of Banach spaces containing pairs of $1$-complemented subspaces whose intersection is not $1$-complemented (and therefore does not have a minimal envelope) - see W.B. Johnson's comment in Mathoverflow \cite{johnson-overflow} 
for a simple proof using pushout. However in reflexive spaces with strictly convex norm we have a positive result, which was also indicated to us by W.B. Johnson.

\begin{proposition}\label{inter}  Let $X$ be a reflexive space with strictly convex norm. Then any intersection of a family $(Y_i)_{i \in I}$ of $1$-complemented subspaces of $X$ must be $1$-complemented.
\end{proposition}

\begin{proof}  Pick a family $(p_i)_{i \in I}$ such that $p_i$ is a contractive projection on $Y_i$ for each $i$.  Given any  non-empty finite $F \subseteq I$, strict convexity implies that the average $T_F:=(1/\#F)\sum_{i\in F}p_i$ of the projections $p_i$ for $i \in F$ has $\bigcap_{i \in F}Y_i$ as its space ${\rm Fix}(T_F)$ of fixed points.
 Consequently,  the Yosida's mean ergodic theorem for reflexive spaces and power-bounded operators \cite{yosida}  implies that the averages of powers of $T_F$ $(1/n \sum_{i=0}^{n-1} T_F^i)_n$ converges SOT to a contractive projection $p_F$ onto $\bigcap_{i \in F}Y_i$.
 By reflexivity we can pick a subnet of $(p_F)_{F\con I,\text{ finite}} $ that converges WOT. The limit will be a contractive projection onto $\bigcap_{i \in I}Y_i$. \end{proof}

 \begin{remark}\label{futureuse} 
  Let us note from the proof of Proposition \ref{inter} that if $X$ is reflexive strictly convex and $p_i$ is a contractive projection onto $Y_i$ for each $i \in I$, then a contractive projection onto $\bigcap_{i \in I}Y_i$ may be chosen in the WOT-closure of the convex hull of $\{p_i\}_{ i \in I}$.
 \end{remark}

\begin{definition}[Minimal envelope of a subspace]
Let $X$ be a Banach space, and $A$ be a subset of $X$. We let
$\env_\mr{min}(A)$ be the envelope defined as the smallest $1$-complemented subspace of $X$ containing $A$, when such subspace exists.
\end{definition}

\begin{remark}\label{isomenvmin} 
When $\env_{\rm min}(A)$ exists,  $\env_{\rm min}(TA)=T(\env_{\rm min}(A))$ for every  $T\in \isome(X)$. 
\end{remark}

From Proposition \ref{inter}, the minimal envelope is always well-defined in reflexive spaces with strictly convex norm.
Under the additional hypothesis that $X$ is strictly convex, then we also have that the contractive projection on a $1$-complemented subspace is unique \cite{CoS}.
In such spaces, the duality mapping plays a central role in understanding the contractive projection onto the minimal envelope.

\begin{definition} For a reflexive space $X$  such that $X, X^*$ are both strictly convex, we  let $J:\mr{Sph}(X) \mapsto \mr{Sph}({X^*})$ be the duality mapping, i.e. $\langle Jx,x\rangle=1$ for every $x \in \mr{Sph}(X)$. 
\end{definition}
It is convenient to extend $J$ to the whole of $X$ by homogeneity.   We recommend   the survey \cite{Beatasurvey}, and also  \cite[Section 4]{FR}, to better understand the relationship betweem $J$ and the isometry groups. In particular, there it is observed that $J$ is well defined, and that it is a  bijection between $\mr{Sph}(X)$ and $\mr{Sph}(X^*)$ whose inverse is the corrresponding duality mapping $J_*$ between $\mr{Sph}({X^*})$ and $\mr{Sph}(X)$; in addition, when $X$ and $X^*$ are assumed   to be locally uniformly rotund, then $J$  is an homeomorphism; finally, $JT={T^*}^{-1}J$ whenever $T \in \isome(X)$.

The next relations between $J$ and $1$-complemented subspaces are due to Calvert \cite{Calvert}, and Cohen and Sullivan \cite{CoS}, see \cite[Thm 5.5 and Thm. 5.6.]{Beatasurvey}.  
\begin{theorem} Let $X$ be reflexive strictly convex and with strictly convex dual.
A closed subspace $Y$ of $X$ is 1-complemented in $X$ if and only if $JY$ is a linear subspace of $X^*$. In this case there is a unique contractive projection of $X$ onto $Y$  which is associated to the decomposition
$X=Y \oplus (JY)^{\perp}$, and if
 $JY$ is closed then it is $1$-complemented, associated to
$X^*=JY \oplus (Y)^{\perp}.$
  \end{theorem}

%
%
%
%
%
%
%
%
%
%
%

Assume that the space $X$ is reflexive and both $X$ and $X^*$ have strictly convex norms. 
Of special interest is the  minimal  (contractive) projection associated to $\env_\mr{min}(Y)$.  Recall that a  {\em minimal contractive projection} on a Banach space is a contractive projection which is minimal with respect to the usual order $p \leq q$ i.e. if and only if $pq=qp=p$.  As we mentioned above under the current hypothesis on $X$, the intersection of 1-complemented subspaces is again 
 1-complemented,  but it is not so obvious how to find (if exists) a minimal contractive projection for a 1-complemented subspace.  To this end we describe the {\em mean ergodic} projection associated to a semigroup of contractions, which may be seen as part of the {\em Jacobs-de Leeuw-Glicksberg} (or JdLG) theory, see \cite[Chapters 8 and 16]{EFHN} for a reference. Recall that $\mr{Fix}(S)$ denotes the subspace of points which are fixed by all elements of a semigroup $S$ of operators on $X$.  We denote by $S^*$ the corresponding semigroup $\{s^*: s \in S\}$ of operators on $X^*$.

Part of the following description of the decomposition of the mean ergodic projection  was obtained in \cite{FR-isometries} under the restriction that $S$ is an isometry group and that $X^*$ is strictly convex.
We also include a description through the convex hull of the semigroup. Note that the description of the second summand  through duality differs from the classical JdLG description.

\begin{proposition}\label{JdLG}  Assume $X$ is reflexive strictly convex,  and that $S$ is a semigroup of contractions on $X$. Then
\begin{enumerate}[{\rm(1)}]
\item $\mr{Fix}(S)$ is $1$-complemented  by a projection in $\overline{\rm{conv}(S)}^\mr{WOT}$.
\end{enumerate}
If furthermore $X^*$ is assumed strictly convex, then
\begin{enumerate}[{\rm(2)}]
\item  The contractive projection onto $\mr{Fix}(S)$ is the unique minimal projection $p_S$ in $\overline{\rm{conv}(S)}^\mr{WOT}$, and we have the formula
$$p_{S}(x)= {\rm the\ unique\ element\ of\ }\mr{Fix}(S) \cap \overline{\rm conv}(S x).$$
\item[{\rm(3)}] The associated decomposition of $X$ is the $S$-invariant decomposition
$$X=\mr{Fix}(S) \oplus J(\mr{Fix}(S))^{\perp}=\mr{Fix}(S) \oplus \mr{Fix}(S^*)^{\perp}.$$ 
  \end{enumerate} \end{proposition}
\begin{proof}  We assume $X$ is reflexive and strictly convex and we prove
 (1). For any $T \in S$, the mean ergodic Theorem states that the averages of powers of $T$ converge to a contractive projection $p_T$ onto ${\rm Fix}(T)$; therefore $p_T$ belongs to the SOT closure of ${\rm conv}(S)$. By Remark \ref{futureuse}, there exists a contractive projection onto $\bigcap_{T \in S}{\rm Fix}(T)={\rm Fix}(S)$, belonging to the WOT-closure of the convex hull of the $p_T$'s, and therefore to the WOT-closure of ${\rm conv}(S)$.

We now assume $X^*$ is strictly convex and prove (2) and (3).
From the fact that $J$ is a bijection it follows immediately that $J(\mr{Fix}(S))=\mr{Fix}(S^*)$ is a linear subspace. Therefore the $1$-complementation of $\mr{Fix}(S)$ by a unique projection (called $q$ for now) satisfying decomposition (3) follow from previously mentioned \cite[Theorems 5.5 and 5.6.]{Beatasurvey}. 

By the  mean ergodic Theorem \cite[Theorem 8.34]{EFHN}, we know that $\overline{S}^\mr{WOT}$ is mean ergodic, and therefore by \cite[Theorem 8.33 b)]{EFHN}   determines uniquely the so-called mean ergodic projection onto $\rm{Fix}(S)=\rm{Fix}(\overline{S}^\mr{WOT})$ called $p$ for now, which is characterized by $px \in \overline{\rm conv}(S x)$ and $pT=Tp=p$ for all $T \in S$.
Strict convexity of $X^*$ imply by \cite{CoS} that $p=q$.

 The uniqueness and existence of the minimal idempotent $p_S$ 
 in  $\overline{{\rm conv}(S)}^\mr{WOT}$ is in \cite[Theorem 16.20]{EFHN}. 
Since $pT=Tp=p$ for all $T \in S$ it is immediate that this also holds for all $T \in \overline{{\rm conv}(S)}^\mr{WOT}$ and therefore for $p_S$, i.e. $p \leq p_S$. To conclude that $p=p_S$ it is now enough to observe that $p \in \overline{{\rm conv}(S)}^\mr{WOT}$ by (1).
\end{proof}

\begin{remark} 
 The condition that $X^*$ strictly convex is necessary for the uniqueness in (2) and the decomposition in (3). Indeed
 if $z$ is a point of $X$  of norm 1 admitting two norming functionals $\phi_0$ and $\phi_1$, then consider the two contractive projections onto the span of $z$ defined by $p_i(x)=\phi_i(x)z$, and let
 $S=\{p_0, p_1\}$. Note that ${\rm Fix}(S)=[z]$ and both $p_0$ and $p_1$ are projections onto it. Also for $i=0,1$, 
 $Fix(p_i^*)=[\phi_i]$,
 therefore ${\rm Fix}(S^*)=\{0\}$ and ${\rm Fix}(S^*)^{\perp}=X$.
\end{remark}

\subsection{The algebraic envelope}
  Recall that given a semigroup $S$ and a subset $Y$ of $X$,  the (point) $S$-stabilizer  $\mr{Stab}_S(Y)$  of $Y$  is the collection of $s\in S$ that acts as the identity on $Y$. Inspired by the decomposition in Proposition  \ref{JdLG} we define the following.

\begin{definition}[$S$-algebraic envelope] 
Given a subspace $Y$ of $X$ and a semigroup $S$ of operators on $X$, the 
{\em $S$-algebraic envelope} of $Y$ is defined as
$$\env_{S}^\mr{alg}(Y)={\rm Fix}(\mr{Stab}_S(Y)).$$
\end{definition}
In other words, the $S$-algebraic envelope of $Y$ is the subspace of points which are fixed by all operators of $S$ fixing each point of $Y$.

It is worth noting that if $J$ is a bijection between the unit spheres of $X$ and $X^*$, and $S$ a semigroup of contractions, then 
$(\mr{Stab}_S(Y))^*:=\conj{s^*}{s \in \mr{Stab}_S(Y)}=\mr{Stab}_{S^*}({JY})$. Note also that $\overline{\mr{Stab}_S(Y)}^{\mr{
WOT}} \subseteq \mr{Stab}_{\overline{S}^\mr{WOT}}(Y)$, but the equality does not seem to hold in general; unless $S$ is WOT-closed, in which case
$\mr{Stab}_{S}(Y)$ is also WOT-closed. The next is a consequence of Proposition 
 \ref{JdLG} and it relates the minimal and algebraic envelope.

\begin{proposition}\label{JdLG-alg}  Assume $X$ is reflexive strictly convex,   $S$ is a semigroup of contractions on $X$, and $Y$ is a closed subspace of $X$. Then
$\env_S^\mr{alg}(Y)$ is  $1$-complemented in $X$ by a contractive projection in
$\overline{{\rm conv}(\stab_S(Y))}^\mr{WOT}$.
If furthermore $X^*$ is strictly convex, then holds the $\stab_S(Y)$-in\-va\-riant decomposition
$$X=\env^\mr{alg}_S(Y) \oplus \env^\mr{alg}_{S^*}(JY)^{\perp},$$ 
and the  projection upon the first summand is the unique minimal projection of $\overline{{\rm conv}(\stab_S(Y))}^\mr{WOT}$.
\end{proposition}

\begin{proof} This is a consequence of Proposition \ref{JdLG} and of the fact that  $\env^\mr{alg}_S(Y)=\mr{Fix}(\stab_S(Y))$ and
that 
$\env^\mr{alg}_{S^*}(JY)=\mr{Fix}(\stab_{S^*}({JY}))=\mr{Fix}((\stab_S(Y))^*)=J(\mr{Fix}(\stab_S(Y)))
=J(\env^\mr{alg}_S(Y))$.
\end{proof}

\begin{remark}
  If $T \in S$ is an automorphism then $\env_S^\mr{alg}(T(Y))=T(\env_S^\mr{alg}(Y))$: note that ${\rm Stab}_S{TY}=T {\rm Stab}_S(Y) T^{-1}$
and that 
$ {\rm Fix}(TST^{-1}) =T({\rm Fix}(S)).$

\end{remark}

Even though the equality ${\rm Fix}(S)={\rm Fix}({\rm conv}(S))$  always holds, the algebraic envelopes of a semigroup $S$ and of its convex hull do not seem to coincide in general. Indeed
we have ${\rm conv}({\rm Stab}_S(Y)) \subseteq{\rm Stab}_{{\rm conv}(S)}(Y)$ but no reason to assume equality, unless under additional hypotheses:

\begin{proposition}\label{algconvG}
Let $X$ be a space with strictly convex norm, let $S$ be a semigroup of contractions on $X$, and let $Y$ be a subspace of $X$. Then 
$$\env_S^{\mr{alg}}(Y)=\env_{\mr{conv}(S)}^{\mr{alg}}(Y).$$
\end{proposition}

\begin{proof}It is enough to prove
that ${\rm conv}({\rm Stab}_S(Y)) ={\rm Stab}_{{\rm conv}(S)}(Y)$, the direct inclusion of which always holds. 
Assume $T \in \mr{conv}(S)$ acts as the identity on $Y$. Writing $T$ as a convex combination
$\sum_k \lambda_k s_k$, $s_k \in S$, we see that for any $y \in Y$,
$y=\sum_k \lambda_k s_k (y)$. By strict convexity and the fact that the $s_k$'s are contractions, we deduce $s_k(y)=y$ for every $k$ and $y \in Y$, i.e. $s_k \in {\rm Stab}_S(Y)$. This means that $T$ belongs to ${\rm conv}({\rm Stab}_S(Y))$.  We have proved that
${\rm conv}({\rm Stab}_S(Y)) \supseteq {\rm Stab}_{{\rm conv}(S)}(Y)$.
\end{proof}

\begin{definition}[Algebraic envelope] The {\em algebraic envelope} of a subspace $Y$ of $X$, denoted by $\env^\mr{alg}(Y)$, is the subspace $\env^\mr{alg}_{\isome (X)}(Y)$. 
\end{definition}

\subsection{Korovkin and isometric envelopes} 
In some aspects the algebraic envelope seems to be ``too rigid" and does not capture topological properties of a semigroup. We introduce a (possibly) smaller topological modification of the algebraic envelope.

\begin{definition}[Korovkin envelope]\label{eq} Let $X$ be a Banach space and $S$ be a bounded semigroup of operators. 
We define the (Korovkin) {\em envelope} $\env_S(A)$ of a subset $A$ of $X$ as the set of $x \in X$ such that whenever a net $(T_i)_{i\in I}$ in  $S$ converges pointwise on $A$ to $\id_A$, then the net $(T_i(x))_{i\in I}$ converges to $x$. 

\end{definition}  
It is straightforward that $Y \mapsto \env_S(Y)$ is an envelope map. Also, we have the following inequalitiles.
 $$ \env_S(Y) \subseteq \env_{\overline{S}^\mr{SOT}}^\mr{alg}(Y) \subseteq \env_S^\mr{alg}(Y).$$
To see this, for suppose that $x$ belong to  $\env_S(Y)$, and assume that $T=\lim_i T_i$,  with each $T_i \in \overline{S}^\mr{SOT}$  and that $T\rest Y=Id_Y$.  Then $(T_i\rest Y)_{i}$ tends SOT to $\id_Y$, and since $x\in \env_S(Y)$, by definition of the Korovkin envelope, we have that  $(T_i x)_i$ converges in norm to $x$, and consequently $Tx=x$.

The name we chose for this envelope  has the following explanation. We recall that we denote by  ${\mathcal L}_1(X)$ the semigroup of contractions on $X$, and if $X$ is a Banach lattice, by ${\mathcal L}_1^+(X)$ the semigroup of positive contractions.
When $S$ is equal to ${\mathcal L}_1(X)$ (resp. ${\mathcal L}_1^+(X)$), results on envelope maps have been obtained by several authors as extensions of {\em Korovkin theorems}, as described for example in \cite{Bernau}, \cite{BerensLorentz}. Namely, if 
$\mathbf T=(T_i)_{i\in I}$ is a net of contractions, the {\em convergence set} $C_{\mathbf T}$ is defined as
$$C_{\mathbf T}:=\{x\in X\,:\, (T_i(x))_{i\in I} {\rm \ converges\ to\ } x\},$$
and if $A$ is a subset of $X$, the {\em shadow} $\Gamma_1(A)$  (resp. $\Gamma_1^{+}(A)$)) of $A$ is the intersection of all convergence sets (resp. of positive contractions) containing $A$ (in the 1994 survey of Altomare and Campiti \cite[Def 3.1.1. p 142.]{AltomareCampiti},  the term {\em Korovkin closure} is also used). In other words, $\Gamma_1(A)$ coincides with the envelope
$\env_{{\mathcal L}_1(X)}(A)$ and $\Gamma_1^+(A)$  with 
$\env_{{\mathcal L}_1^+(X)}(A)$.

We use the terminology  {\em Korovkin set} (resp. Korovkin$^{+}$ set) in $X$ for a set $A$ for which $\Gamma_1(A)=X$
(resp. $\Gamma_1^+(A)=X$).
Korovkin Theorem \cite{Korovkin} states that $\{1,t,t^2\}$ is a Korovkin$^{+}$ set in $C(0,1)$. This was later improved to being a Korovkin set
by Wulbert \cite{Wulbert} and independently, Saskin \cite{Saskin}. 
Bernau \cite{Bernau} proves that
$\{1,t\}$ is a Korovkin set in $L_p(0,1), 1<p<+\infty$, and Berens-Lorentz \cite{BerensLorentz} for $p=1$. In \cite{Calvert}, Calvert obtains that if $X$ is a reflexive space with locally uniformly convex norm and dual norm, then convergence sets are ranges of linear contractive projections, and the envelope $\Gamma_1(A)$  is the smallest $1$-complemented subspace containing $A$, i.e. $\Gamma_1(A)=\env_\mr{min}(A)$, \cite{Calvert} Corollary 2. We shall revisit these results in the present section.

\begin{definition}[Isometric envelope]
When $G$ is the isometry group of $(X,\|\cdot\|)$ then $\env_G(Y)$ is called the {\rm (isometric) envelope} of $Y$ in $X$, and we denote it by $\env_{\|\cdot\|}(Y)$, or simply $\env(Y)$, when the norm on $X$ is clear from the context. 
\end{definition}

It will be relevant to give a name to those subspaces whose envelope coincides with the whole space:

\begin{definition}[Full envelope, full subspace] A subspace  $Y$   of $X$ has {\em full envelope}, or it is a {\em full subspace} of $X$, when  $\env(Y)=X$.  \end{definition}

When $G$ is a bounded group of isomorphisms, we
 have the following   reformulation of the definition of the envelope $\env_G(A)$.
Because of item c) in the definition of envelope, from now on we only consider envelopes of closed subspaces $Y$ of a Banach space $X$.

\prop\label{pfpfpfprrr} Let $X$ be a Banach space and $Y$ a closed subspace of $X$.
If $G$ is a bounded group of {\em isomorphisms} on $X$, then a point $x\in X$ belongs to   $\env_G(Y)$   if and only if whenever a net $(T_i)_{i\in I}$ in $G$ converges pointwise  on $Y$, then $(T_i (x))_{i\in I}$ also converges.

\fprop 

\begin{proof} Let $G$ be a bounded group of isomorphisms on $X$. Assume that $x$ satisfies the last condition, and let us see that $x\in \env_G(Y)$, so suppose that  $(T_i)_{i\in I}$ is a net that converges pointwise  on $Y$ to $\id_Y$.  We consider the set $J:=I\times \{0,1\}$ lexicographically ordered, and the net $(U_{(i,j)})_{(i,j)\in J}$, $U_{(i,0)}:= \id$ and $U_{(i,1)}:=T_i$ for all $i$. It follows that $(U_{(i,j)})_{(i,j)}$ pointwise converges in $Y$ to $\id_Y$, so by hypothesis, $(U_{(i,j)}(x))_{(i,j)}$ also converges, and the limit must be $x$. This means in particular that $(U_i(x))_i$ converges to $x$.

Suppose now that $x$ does not satisfy the last condition. We can find a net  $(T_i)_{i\in I}$ in $G$ converging pointwise on $Y$ and some $\vep>0$ such that the subset $J\con I$ of those $i$ such that there are $j_i, k_i \ge i$ such that $\nrm{T_{j_i}(x)-T_{k_i}(x)}\ge \vep$. This implies that, $\nrm{x- T_{j_i}^{-1}(T_{k_i}(x))}\ge \vep /\sup_{T\in G} \nrm{T} $ for every $i\in J$. For each $i\in J$ let $U_i:= T_{j_i}^{-1}\circ T_{k_i}\in G$. Then the net $(U_i)_{i\in J}$ in $G$ converges pointwise on $Y$ to the identity of $Y$ but $(U_i(x))_{i\in J}$ does not converge to $x$, so $x\notin \env_G(Y)$.   
\end{proof}

%
%
 So, $\env_G(Y)$ is the maximum superset  $Z$ of $Y$ so that  pointwise convergence on $Y$ of elements of $G$  (to a map $t: Y \rightarrow X$ which does not necessarily extend to an element of $G$)   implies pointwise convergence of these elements on $Z$ (necessarily to an extension of $t$).

Note that when $S \subseteq S'$ are bounded semigroups then $\env_{S'}(Y) \subseteq \env_S(Y)$, and also that whenever $T \in S$ is an automorphism, then $\env_S(TY)=T(\env_S(Y))$.

\begin{example}\label{hilbert} If $X=H$ is a Hilbert space and $Y$ is a closed subspace of it  then $\env(Y)=Y$. Indeed we can decompose $H$ as $Y \oplus Y^{\perp}$ and  if $x$ is not in $Y$, then let $z={\rm proj}_{Y^{\perp}} x \neq 0$ and define $T_n=\id \oplus -\id$. This proves that $x$ is not in $\env(Y)$. \end{example}
 
 A group of isomorphisms $G$ on a space $X$ is said to be trivial if all elements of $G$ are multiples of the identity map.

\begin{example} If the isometry group $\isome(X)$ is trivial,  then  any  subspace $Y$ of $X$ of dimension at least $1$ has envelope equal to $X$:  if  a net $(T_i)_{i\in I}$ converges pointwise on $Y$, then each $T_i=\lambda_i\id_X$ with $(\lambda_i)_{i\in I}$ a converging net; thus, $(T_i)_{i\in I}$ converges pointwise on the whole of $X$.
\end{example}

  As in the case of the algebraic envelope, in general the Korovkin envelope of a semigroup $S$ does not seem to need to coincide with the envelope of its WOT-closure or of its convex hull. There is however such an identification under mild topological conditions on the norm. The hypothesis of the norm being locally uniformly convex (LUR) in the next proposition   should be related to  the fact that the Korovkin envelope is of a topological nature. 

\begin{proposition}\label{convG}  
Let $X$ be a space with a LUR norm and let $S$ be a semigroup of contractions on $X$. Let $Y$ be a subspace of $X$. Then 
$$\env_S(Y)=\env_{\overline{\mr{conv}(S)}^{\mr{WOT}}}(Y).$$
\end{proposition}

\begin{proof} 
Only the direct inclusion is not trivial.
Let $x$ be normalized in $\env_S(Y)$ and fix $\varepsilon>0$.
By definition of $\env_S(Y)$, there exist normalized
$y_1,\ldots,y_n$  in $Y$ and $\de>0$  such that for any $g \in S$,
$$\text{ if $\|g(y_i)-y_i\| \leq \de$ for every $i=1,\ldots,n$, then $\|gx-x\| \leq \varepsilon.$} $$
We use the LUR property  of the norm to choose for each $i=1,\ldots,n$ some $\delta_i>0$  such that  if $\nrm{z}\le 1$ is such that  $\|z+y_i\| \geq 2-\delta_i$ then $ \|z-y_i\| \leq \de$. 
Let $\beta>0$ be small enough so that $\sqrt{2\beta} \leq \min_i \delta_i$, $\beta \leq \varepsilon$
and $2n\sqrt{2\beta} \leq \varepsilon$.
We claim that whenever $T \in \overline{\mr{conv}(S)}^\mr{WOT}$ satisfies that
$\max_i\|Ty_i-y_i\| \leq \beta$, then it follows that 
$\|Tx-x\| \leq 3\varepsilon$. This implies that whenever $(T_i)_{i\in I}$ is a net in $\overline{\mr{conv}(S)}^\mr{WOT}$ converging pointwise to the identity on $Y$, then $(T_i(x))_{i\in I}$ converges to $x$, and this means that $x$ belongs to $\env_{\overline{\mr{conv}(S)}^\mr{WOT}}(Y)$. This therefore proves the required direct inclusion.

We now give the proof of the claim. For $i=1,\ldots,n$ let $\phi_i$ be a normalized functional norming $y_i$, and let $\phi_0$ be a normalized functional norming $Tx-x$. Denote $y_0:=x$ and consider a convex combination
$\sum_{k \in K} \lambda_k g_k$ with each $g_k \in S$, such that $|\langle\phi_i,(T-\sum_{k \in K} \lambda_k g_k) y_i\rangle| \leq \beta$ for $i=0,\ldots,n$.  Note that for every $i=1,\ldots,n$,
\begin{equation}\label{oi3j4iorj4r4}
    1-2\beta \leq \phi_i(T(y_i))-\beta \leq \sum_{k \in K} \lambda_k \phi_i(g_k (y_i)).
\end{equation}
For each $i=1,\dots,n$, let $A_i$ be the set of indices $k\in K$ such that $\phi_i(g_k (y_i)) \geq 1-\sqrt{2\beta}$ and let $B_i=K \setminus A_i$ be the complement of $A_i$.
Note that for $k \in A_i$, we have that $\|g_k (y_i) - y_i\| \leq \de$.
Also from the above equation \eqref{oi3j4iorj4r4} we have
$$1-2\beta \leq \sum_{k \in A_i} \lambda_k +(1-\sqrt{2\beta})\sum_{k \in B_i} \lambda_k,$$
from which it is immediate to deduce 
$$\sum_{k \in B_i}\lambda_k \leq \sqrt{2\beta}.$$
Finally let $A=\bigcap_{i=1}^n A_i$ and $B=\bigcup_{i=1}^n B_i$.
Whenever $k \in A$, we have $\|g_k y_i - y_i\| \leq \de$ for all $i=1,\ldots,n$, and
therefore $\|g_k x -x \| \leq \varepsilon$, by the choice of $\de$.
Using this estimate, we compute
$$
\|Tx-x\|=\phi_0(Tx-x) \leq \beta+\sum_{k \in K}\lambda_k \|g_k x-x\|
\leq \beta+(\sum_{k \in A}\lambda_k)\varepsilon+ 2\sum_{i=1}^n \sum_{k \in B_i}
\lambda_k \leq 2\varepsilon+2n\sqrt{\beta} \leq 3\varepsilon,$$
and this concludes the proof of the claim.
\end{proof}

\begin{remark}\label{remrem} An interesting consequence of Proposition \ref{convG} is that if $Y$ is a subspace of an LUR space $X$ and $S$ a semigroup of contractions, then $\env_S(Y)$ is contained in the range of any contractive projection $p \in \overline{{\rm conv}(S)}^{\rm WOT}$ acting as the identity on $Y$. Indeed 
from $p\rest Y =Id_Y$ it follows that $px=x$ for any $x \in \env_{\overline{{\rm conv}(S)}^{\rm WOT}}(Y)=\env_S(Y)$. \end{remark}

We now show that under natural conditions on the space (such as reflexivity), the Korovkin envelope may be seen as an algebraic envelope.
This is important to obtain a JdLG decomposition associated to this envelope. We shall use the following facts. If $X$ is reflexive (actually the point of continuity property is enough) and if $G$ is a bounded group of automorphisms on $X$, then the weak and norm topologies coincide on each $G$-orbit of  a non-zero point $x_0$ of $X$  \cite[ Theorem 2.5]{M}, applied for $G$ equipped with the topology of weak convergence in the point $x_0$.  
Also it is a classical and easy fact that if the norm on $X$ is locally uniformly rotund (LUR), then weak convergence of a net on the unit ball of $X$ to a point of the sphere, implies strong convergence to this point. See \cite[Remark 2.4.]{AFGR}  for more details on both these facts.

\begin{proposition}\label{WOT} Let  $X$ be a reflexive space. Assume that either $S$ is a bounded group of isomorphisms on $X$, or that $X$ is  LUR and $S$ is a semigroup of contractions. 
Then
 $$\env_S(Y)=\env^\mr{alg}_{\overline{S}^\mr{WOT}}(Y).$$
In particular the two envelopes coincide as soon as $S$ is WOT-closed.
\end{proposition}
\begin{proof} In the first case,
we may assume that $S$ is a group of isometries for some renorming of $X$.  Therefore in either case, for any net $(T_i)_i$ in $S$, 
and normalized vector $x \in X$, weak convergence of $(T_i x)_i$ to $x$ implies norm convergence.   Assume some  $x$ does not belong to $\env_S(Y)$ and let $(T_i)_i$ converge pointwise to $\id_{Y}$ on $Y$ such that $(T_i x)_i$ does not converge to $x$ in norm and therefore does not converge weakly to $x$.  Since reflexivity of $X$ is equivalent to saying that the space of contractions with the WOT is compact, there is convergent subnet for $(T_i)_i$. W.l.o.g., we may assume that $(T_i)_i$ converges WOT to some $T$, and  note that $T$ is the identity on $Y$, but that $Tx$ is not equal to $x$. This proves that $x$ does not belong to 
$\env^\mr{alg}_{\overline{S}^\mr{WOT}}(Y)$.

Likewise, let $x \in \env_S(Y)$ and $T\rest{Y}=\id_Y$ for $T \in \overline{S}^\mr{WOT}$, then let $(T_i)_i$ a net in $S$ converging WOT to $T$, so that $(T_i)_i$ converges WOT to $\id$ on $Y$. Then the convergence of $(T_i)_{|Y}$ to $\id_Y$ is SOT and so $(T_i x)_i$ converges to $x$, so $Tx=x$, and consequently $x \in \env^\mr{alg}_{\overline{S}^\mr{WOT}}(Y)$.
\end{proof}

The next proposition will allow us to compare the Korovkin and the minimal envelopes. 
We have the following decomposition for the Korovkin envelope that should be compared with Proposition  \ref{JdLG-alg}.

\begin{proposition}\label{JdLG-kor}  Let $X$ be  reflexive strictly convex, $S$ a semigroup of contractions on $X$, and  let $Y$ be a closed subspace of $X$.  Assume additionally that either (i) $S$ is a group of isometries on $X$, or  that (ii)  $X$ is LUR.    Then
\begin{enumerate}[\rm(1)]
 \item
$\env_S(Y)$ is a $1$-complemented subspace of $X$; in case (ii)  $\env_S(Y)$ is the smallest superspace of $Y$ complemented by a projection in
$\overline{{\rm conv}(S)}^\mr{WOT}$.
\item
If furthermore $X^*$ is strictly convex, then holds the 
${\rm Stab}_{\overline{\rm S}^\mr{WOT}}({Y})$-in\-va\-riant de\-com\-po\-sition
\begin{equation}\label{hu4h5ur34hiu5r34}
X=\env_S(Y) \oplus \env_S(JY)^{\perp},
\end{equation}
and in case (ii) the  projection on the first summand is the unique minimal projection of  the semigroup $\overline{{\rm conv}(S)}^{\rm{WOT}}$ acting as the identity on $Y$.
\end{enumerate}
\end{proposition}

\begin{proof} The $1$-complementation in (1) and the decomposition in \eqref{hu4h5ur34hiu5r34}  is a consequence of Proposition  \ref{JdLG-alg} and  the relation between the algebraic and the Korovkin envelope in   Proposition \ref{WOT}. Regarding the projection in case (1) {\it(ii)}, note that as a first step the set it belongs to is $\overline{{\rm conv}({\rm Stab}_{\overline{S}^\mr{WOT}}(Y))}^\mr{WOT}$. Under the LUR property and by
 Proposition \ref{convG}, this set is the same as
 $\overline{{\rm conv}({\rm Stab}_{\overline{{\rm conv}(S)}^\mr{WOT}}(Y))}^\mr{WOT}$.
 But it is also the same as ${\rm Stab}_{\overline{{\rm conv}(S)}^\mr{WOT}}(Y)$  since this set is convex and WOT-closed. This proves that  $\env_S(Y)$ is  complemented by a projection in
$\overline{{\rm conv}(S)}^\mr{WOT}$, and that it is the smallest among those containing $Y$ is Remark \ref{remrem}. The statement about the projection in (2){\it(ii)} follows from the one in case (1){\it(ii)} and strict convexity of the dual.
 \end{proof}

As a consequence we spell out the following (note that (3) is an improvement of Calvert's result \cite{Calvert} on the Korovkin envelope, since there it is needed $X^*$ to be LUR and we impose no requirement on it).

\begin{proposition}\label{ghfddddd}  Let $X$ be reflexive strictly convex and let $Y$ be a closed subspace of $X$. Then
\begin{enumerate}[{\rm(1)}]
    \item the isometric envelope $\env(Y)$ is $1$-complemented.
\end{enumerate}
If furthermore $X$ is LUR then
\begin{enumerate}[{\rm(1)}]\addtocounter{enumi}{1}
\item  $\env(Y)$ is the smallest superspace of $Y$ complemented by a projection in  $\overline{{\rm conv}(\isome(X))}^{\rm{WOT}}$; if $X^*$ is strictly convex, then this projection is the minimal projection in 
$\overline{{\rm conv}(\isome(X))}^{\rm{WOT}}$
acting as the identity on $Y$.
\item The Korovkin envelope $\env_{\mc L_1(X)}(Y)$ coincides with the minimal envelope $\env_\mr{min}(Y)$.
\end{enumerate}
\end{proposition}

\begin{proof}
 All items follow from Proposition \ref{JdLG-kor}. 
\end{proof}


\begin{remark}\label{remarkiii} Note that under the assumption that $X$ is  reflexive strictly convex,  we have 
$$\env_\mr{min}(Y) \subseteq \env(Y) \subseteq \env^\mr{alg}(Y)$$.

We do not know of a general characterization of situations where $\env(Y)$ and
$\env^\mr{alg}(Y)$ must coincide.  However when $X=L_p$, $p\neq 2$, and $Y$ is a unital subspace, then we will see in Theorem \ref{62} that $\env_\mr{min}(Y)= \env(Y)=L_p([0,1],\Sig_Y)$, where $\Sig_Y$ is the minimal subalgebra of the Borel algebra $\mc B([0,1])$ making all functions in $Y$ measurable. And in this case $\env(Y)=\env^\mr{alg}(Y)$ exactly when $\Sig_Y$ is a {\em fixed-point subalgebra} (see Remark \ref{lknjkio855}).    
\end{remark}

\begin{remark}
In \cite{BeMa} Beauzamy and Maurey consider a non-linear procedure associating to a subset $M$ the set $M\con\min M$ of its minimal points in a metric sense. When $X$ is reflexive strictly convex and with strictly convex dual, it follows from their results that the minimal envelope of a subspace may be obtained through iteration of their "min" procedure.  In this way they recover the results of Bernau \cite{Bernau} and Calvert \cite{Calvert} mentioned earlier.
\end{remark}

\begin{remark} 
It may also be useful to spell out, for $X$ reflexive LUR:  $Y=\env(Y)$ if and only if
$Y$ is complemented by a contractive projection in $\overline{{\rm conv}(\isome(X))}^{\rm WOT}$. This characterizes the subspaces which are isometric envelopes as those which are ranges of projections in the ``isometric hull" $\overline{{\rm conv}(\isome(X))}^{\rm WOT}$.
This fully describes the isometric envelope map, as observed at the beginning of this section just after the definition of abstract envelope maps.
\end{remark}

\begin{proposition}\label{uhiu4iur4rt4}  Assume $X$ is separable reflexive, let $G$ be a subgroup of $\isome(X)$, and let $Y$ be a subspace of $X$. Then $\env_G(Y)$ is $1$-complemented in $X$ by a projection in $\overline{{\rm conv}(G)}^{\rm WOT}$.
\end{proposition}

\begin{proof}  By Lancien \cite{Lancien} we may renorm $X$ with a norm $\nrm{\cdot}$ which is $1+1/n$-equivalent to the original norm and so that $X$ is LUR and $G \subseteq \isome(X,\nrm{\cdot})$. It follows from Proposition \ref{JdLG-kor} (1)(ii) that $\env_G(Y)$ is $1$-complemented with respect to $\nrm{\cdot}$ and therefore $(1+1/n)^2$-complemented in the original norm, by a projection in $\overline{{\rm conv}(G)}^{\rm WOT}$. A WOT-cluster point of the sequence of associated projections $P_n$ is a norm $1$-projection onto $\env_G(Y)$ belonging to $\overline{{\rm conv}(G)}^{\rm WOT}$.
\end{proof}

\begin{corollary} The Hilbert space is the only separable reflexive space for which each closed subspace $Y$ ($2$-dimensional subspace is enough) is equal to its envelope $\env(Y)$.
\end{corollary}

\begin{proof} Indeed it is a known result by Kakutani \cite{K} that X is Hilbertian if and only if all subspaces are $1$-complemented (in fact it is enough to require that all $2$-dimensional subspaces are 1-complemented); in the complex case this is due to Bohnenblust \cite{Bohnenblust}. \end{proof}

%
%
%
%
%
%
%
%

\subsection{Strong isometric embeddings}
Closely related to the isometric envelope, as well as a motivation for its definition, are the embeddings who are limits or restriction of (global) isometries. 
\begin{definition} Given a  Banach space $X$, and a subspace $Y\con X$, one defines the collection
${\rm Emb}_\mr{ext}(Y,X)$ of {\em extendable} isometric embeddings
as
$${\rm Emb}_\mr{ext}(Y,X):={\rm Isom}(X)\rest{Y} \subseteq {\rm Emb}(Y,X),$$ and
the collection ${\rm Emb}_\mr{s}(Y,X)$  of {\em strong} isometric embeddings as
$${\rm Emb}_\mr{s}(Y,X):=\overline{{\rm Isom}(X)\rest{Y}}^\mr{SOT} \subseteq {\rm Emb}(Y,X).$$
\end{definition}

Note that ${\rm Emb}_\mr{s}(Y,X)$ is by definition SOT-closed in ${\rm Emb}(Y,X)$, hence a motivation for the name.

When $X$ is a Hilbert space, then ${\rm Emb}_\mr{ext}(F,X)={\rm Emb}(F,X)$ for all finite dimensional subspaces $F$ of $X$. Spaces with these properties are called {\em ultrahomogeneous}.  The Hilbert space is the only known separable example; non separable examples appear in \cite{ACCGM} and \cite{FLMT}.  When ${\rm Emb}_\mr{s}(F,X)= {\rm Emb}(F,X)$ for every finite dimensional $F$, the space is called approximately ultrahomogeneous (see \cite{FLMT}). Separable examples are the Hilbert spaces, all $L_p[0,1]$ for $p\notin 4 + 2\N$, or the Gurarij space.


We shall sometimes use the word
 {\em partial isometry}  on a space $X$ to mean a surjective isometry $t$ between subspaces $Y$ and $Z$ of $X$.  A {\em strong partial isometry} is a partial isometry $t: Y \rightarrow Z$ such that $i_{ZX} \circ t$ belongs to ${\rm Emb}_{s}(Y,X)$, where $i_{ZX}$ is the inclusion map $Z \rightarrow X$.

\begin{lemma} The inverse of a strong partial isometry is a strong partial isometry. When defined, the composition of two strong partial isometries is a strong partial isometry. \end{lemma}

\begin{proof}   Let $\gamma: Y \rightarrow Z$ be a strong partial isometry. Given $z_1,\dots, z_n\in Z$ of norm 1, let $y_j:=\ga^{-1} z_j$ for each $1\le j\le n$. Let $g\in \mr{Isom}(X)$ be such
that $\nrm{g(y_j)-\ga(y_j)}\le \vep$ for every $1\le j\le n$. Then for each $1\le j\le n$ one has that $\nrm{\ga^{-1}(z_j)-g^{-1}(z_j)}=\nrm{y_j-g^{-1}(\ga(y_j))}=\nrm{g^{-1}gy_j-g^{-1}\ga(y_j)}=\nrm{gy_j -\ga y_j}\le \vep$. Similarly one proves the second part of the statement.
\end{proof}

We may therefore define an equivalence relation between subspaces of $X$ as follows: $Y$ and $Z$ are {\em strongly isometric} if there exists an strong partial isometry between $Y$ and $Z$. The next proposition substantiates the idea that isometric envelopes are well preserved under strong partial isometries.

\begin{proposition}\label{extension} 
Let $X$ be a Banach space and $Y$  a closed subspace of $X$. Then
\begin{enumerate}[\rm (1)]
\item Every $t \in {\rm Emb}_{s}(Y,X)$ extends uniquely to some $\tilde{t} \in {\rm Emb}_{s}(\env(Y),X)$.
\item The image $\tilde{t}(\env(Y))$ is equal to $\env(tY)$.  
\end{enumerate}
In particular if $Y$ and $Z$ are  strongly isometric subspaces of $X$, then so are $\env(Y)$ and $\env(Z)$. Furthermore,
\begin{enumerate}[\rm (3)]
\item If $Y$ is separable, the map $t \mapsto \tilde{t}$ is SOT-SOT continuous from ${\rm Emb}_{s}(Y,X)$ to ${\rm Emb}_{s}(\env(Y),X)$.
\end{enumerate}
\end{proposition}

\begin{proof}
{\rm (1)}: An extension to $\env(Y)$ is  defined by $\tilde{t}(x)=\lim_{i \in I} T_i(x)$, if $t=\lim_i T_{i}\rest Y$ with $T_i \in {\rm Isom}(X)$. If $t'$ is a strong isometric embedding on $\env(Y)$ defined as pointwise limit of some $(U_i)_i$, extending $t$, then note that since the sequence $(T_i^{-1} U_i)_i$ is pointwise convergent to $\id_Y$ on $Y$ it is also pointwise convergent on $\env(Y)$, implying that $\tilde{t}$ and $t'$ coincide on $\env(Y)$.

{\rm (2)}: We claim that $\tilde{t}(\env(Y)) = \env(tY)$. We first prove  that $\tilde{t}(\env(Y)) \con \env(tY)$, so fix $x\in \tilde{t}(\env(Y))$. Let $(T_i)_{i\in I}$ be a net of isometries of $X$ that defines $\tilde{t}$. Suppose that $(U_j)_{j\in J} $ is an arbitrary sequence on $\isome(X)$ such that $(U_j  t y)_j$ converges for all $y\in Y$. Without loss of generality we may assume that $I=J$: One can use the product ordering $I\times J$  and define $T_{i,j}:=T_i$ and $U_{i,j}:=U_j$.   It follows that $(U_i \circ T_i y)_i$ converges for all $y\in Y$, so, by hypothesis, $(U_i \circ T_i x)_i$ converges, or equivalently $(U_i \tilde{t}(x))_i$ converges. This means that $\tilde{t}(x)\in \env (tY)$.  Now let $u:tY\to Y$ be the inverse of $t$. We know that $u\in \Emb_\mr{s}(TY,X)$, hence, $\tilde{u}\env(tY)\con \env(u t Y)=\env(Y)$. This means that if $x\in \env(tY)$, then $\tilde{u}(x)\in \env(Y)$, and so $x=\tilde{t}(\tilde{u}(x))\in \tilde{t}(\env(Y))$.

{\rm (3)}: By separability of $Y$, ${\rm Emb}(Y,X)$ is a metric space. Assume $(t_n)_n$ converges SOT to $t$ and let $x \in \env(Y)$.
Let $V_n$ be an element of ${\rm Isom}(X)$ such that $d((V_n)\rest {Y},t_n)\leq 1/n$, where $d$ is the SOT-convergence metric on ${\rm Emb}(Y,X)$, and so that $\|V_n x -\tilde{t}_n x\| \leq 1/n$ (by definition of $\tilde{t}_n$). Then note that $V_n$ converges SOT to $t$ on $Y$ and therefore $(V_n x)_n$ converges to $\tilde{t}x$ (by definition of $\tilde{t}$). So, $\tilde{t_n} x$ converges to $\tilde{t} x$ for all $x \in \env(Y)$. 
\end{proof}

\cor\label{ijeriojweiojrewr} The following are equivalent for $Y\con X$:
\begin{enumerate}[\rm(1)]
    \item $Y$ is strongly isometric to a full subspace of $X$.
    \item $\env(Y)$ is strongly isometric to $X$. 
\end{enumerate}

\fcor	
\begin{proof}
Suppose that $t:Y\to X$ is a strong isometry such that $tY$ is a full subspace. it follows from Proposition \ref{extension} {\rm (2)} that $\tilde{t}(\env(Y))=\env(tY)=X$, hence $\tilde{t}: \env(Y)\to X$ is a (surjective) strong isometry. 

 Suppose that $T:\env(Y) \to X$ is a surjective strong isometry. Then $t:=T\rest Y\in \Emb_\mr{s}(Y, X)$, $\tilde{t}=T$, and it follows from Proposition \ref{extension} {\rm (2)} that $X=T(\env(Y))=\env(T(Y))$, so $TY$ is a full subspace of $X$. 
\end{proof}  

The envelope of $Y$ admits an important characterization when $Y$ is assumed separable.

\begin{proposition}\label{extensionbis} Assume $Y$ is a separable  subspace of $X$. Then $\env(Y)$ is the maximum among subspaces $Z$ of $X$ with the property that 
\begin{enumerate}[\rm (1)]
\item
 $Z$ contains $Y$.
\item Any embedding $t$  in ${\rm Emb}_{s}(Y,X)$ extends uniquely  to an embedding $\tilde{t}$ in ${\rm Emb}_{s}(Z,X)$.
\item  The mapping $t \mapsto \tilde{t}$ is SOT-SOT continuous.
\end{enumerate}
\end{proposition}

\begin{proof} The space $\env(Y)$ does satisfy {\rm (1)}, {\rm (2)} and {\rm (3)}. Let now $Z$ be such a subspace and $z \in Z$. If $(T_n)_n$ is a sequence of isometries on $X$ converging SOT on Y, let $t_n=T_{n|Y} \in {\rm Emb}_\mr{ext}(Y,X) \subseteq {\rm Emb}_\mr{s}(Y,X)$ and $t=\lim_n t_n$. By  definition of $Z$, $(T_{n}\rest Z)_n=(\tilde{t}_n)_n$ tends SOT to $\tilde{t}$; therefore $(T_n z)_n$ converges. This proves that $z \in \env(Y)$ and therefore that $Z \subseteq \env(Y)$.
\end{proof}

Propositions \ref{extension} and \ref{extensionbis} lead us to see the class of subspaces of $X$ which are isometric envelopes as a 
``strongly isometric skeleton" of $X$, in the sense that it is a subfamily of subspaces of $X$ ``containing all the information" on strong isometric embeddings inside $X$.

Finally a relevant case of strong isometric embedding defined on a subspace $Y$ are those whose image is equal to $Y$, and may therefore be seen as surjective isometries on $Y$. In general, not all surjective isometries on $Y$ is of this form. However in a class of spaces that we will consider in the next subsection  \ref{oi43iot43t4r} this is the case and we will see that is implies that $Y\con X$ is a $g$-embedding.

\begin{definition}[Strong isometry] Let $X$ be a Banach space and $Y$ a closed subspace of $X$. We let $\isome_\mr{s}(Y)$ be the set of {\em strong isometries} on $Y$, i.e.,  of strong isometric embeddings of $Y$ whose image is $Y$. \end{definition}

Equivalently $t \in \isome(Y)$ is a strong isometry on $Y$ if and if only if it is a surjective isometry on $Y$ which belongs to ${\rm Emb}_s(Y,X)$. In other words, elements of $\isome_\mr{s}(Y)$ are isometries on $Y$ which, as maps with value in $X$, belong to $\Emb_\mr{s}(Y,X)$. In what follows we always consider the SOT.

\begin{lemma} Let $X$ be a Banach space and $Y$ a closed subspace of $X$. 
Then $\isome_\mr{s}(Y)$ is a closed subgroup of $\isome(Y)$. 
\end{lemma}

\begin{proof} It is clearly closed. If $t, u$ are two such elements, with $t=\lim_i T_{i}\rest Y$ 
and $u=\lim_i U_{i}\rest Y$, then $({T_i U_i}\rest {Y})_i=(T_i (U_{i}\rest Y-u) +T_i u)_i$ also tends SOT to $tu$.
Likewise $t^{-1}=\lim_i (T_{i}^{-1})\rest{Y}$.
\end{proof}

\begin{corollary}\label{oij4io3tbjjittre} If $Y$ is a closed subspace of $X$, then $t\in \isome_\mr{s}(Y)\mapsto \tilde{t}\in \isome_\mr{s}(\env(Y))$ is an embedding of topological groups.
\end{corollary}

\begin{proof}
The group embedding property is clear. Let us denote this embedding by $e$. Then let $t_i \in \isome_{s}(Y)$ converge to $t \in \isome_{s}(Y)$,
and let $T_i=e(t_i)$ defined on $\env(Y)$ extending $t_i$. Note that $T_i$ converges pointwise to $t$ on $Y$, so by Proposition \ref{pfpfpfprrr}, $i_{\env(Y),X} T_i$ converges pointwise on $\env(Y)$ to some $T
\in {\rm Emb}_\mr{s}(\env(Y))$ such that $T_{|Y}=t$. By the uniqueness result {\rm(1)} in Proposition \ref{extension}, $T_i$ converges pointwise to $e(t)$. \end{proof}

\subsection{Envelopes in AUH or Fra\"iss\'e spaces}\label{oi43iot43t4r}
Recall that a space $X$ is called ultrahomogeneous (UH)  when   $\Emb_\mr{ext}(Y,X)= \Emb(Y,X)$   for every finite dimensional subspace $Y$ of $X$ ; $X$  is called approximately ultrahomogeneous (AUH) when  $\Emb_\mr{s}(Y,X)= \Emb(Y,X)$ for such subspaces $Y$ of $X$.

A separable \auh space without finite cotype must  be isometric to the Gurarij space. This follows from the fact that two separable Fra\"iss\'e spaces which are finitely representable into each other must be isometric \cite[Theorem 2.19]{FLMT}. In the finite cotype case let us also note that if  a Banach space $X$ is reflexive (or even just Asplund or with the RNP) and  \auh, then, since $X$ is in particular almost transitive, then it is also superreflexive and  $X$ and $X^*$  are uniformly convex \cite{CabelloSanchez}.  

We obtain the following results 
regarding isometric embeddings in \auh spaces, easy to prove directly from the corresponding results for strong isometric embeddings in general spaces.



The next is Proposition \ref{extension} reformulated for \auh spaces:

\begin{proposition}\label{ext}  Assume $X$ is \auh, let $Y$ be a subspace of $X$, let $t \in {\rm Emb}(Y,X)$ be an isometric embedding of $Y$ into $X$. Then
\begin{enumerate}[\rm(1)]
\item $t$ extends uniquely  to an isometric embedding $\tilde{t} \in {\rm Emb}(\env(Y),X)$; and $\tilde{t}(\env(Y))=\env(t(Y))$.
\end{enumerate}
Assume that $Y$ is separable. Then,
\begin{enumerate}[\rm(1)]  \addtocounter{enumi}{1}

\item  the map $t \mapsto \tilde{t}$ is SOT-SOT continuous.
\item   $\env(Y)$ is the maximum among spaces $Y'$ such that any isometric embedding $t: Y \rightarrow X$ extends uniquely to an isometric embedding $\tilde{t}: Y' \rightarrow X$, and such that $t \mapsto \tilde{t}$ is SOT-SOT  continuous. \qed

\end{enumerate}  
\end{proposition}

\begin{definition}[linear $g$-embedding]
We say that a subspace $Y$ of $X$ is a linear $g$-embedding when  is an embedding of topological groups $e:\isome(Y)\to \isome(X)$ such that $e(h)\rest Y= h$ for every $h\in \isome(Y)$.  
\end{definition}
This is   the linear version of the notion of $g$-embedding of a   subspace $M$ of a metric space $L$ introduced in   \cite[Definition 3.1]{Usp} (see also \cite[Definition 5.2.6]{Pes}), that demands the existence of an embedding of topological groups $e:\mc{I}so(M)\to \mc{I}so(L)$ such that $e(h)\rest Y= h$ for every $h\in \mc{I}so(L)$, where $\mc{I}so(N)$ is the group of isometries of $N$.  Recall that the classical Mazur-Ulam Theorem states that the group $\mc{I}so(Y)$ of surjective, not necessarily linear, isometries on $Y$ is the group of the affine ones on $Y$. From this it follows easily that if $e: \isome(Y)\to \isome(X)$ is an embedding of topological groups, then $\bar e: \mc{I}so(Y)\to\mc{I}so(X)$ defined by 
$\bar e(\ga)= e(\ga- \ga(0))+\ga(0)$ is also an embedding as topological groups, and consequently every linear $g$-embedding is a $g$-embedding.

We have the following direct consequence of  Corollary \ref{oij4io3tbjjittre}.
\begin{proposition}\label{extisometry} 
Assume $X$ is \auh and let $Y$ be a subspace of $X$.  Then
\begin{enumerate}[\rm(1)]
\item any surjective isometry $t$ on $Y$ extends uniquely  to a surjective isometry $\tilde{t}$ on $\env(Y)$. 
\item The map $t \mapsto \tilde{t}$ defines a topological (group) embedding of
$\isome(Y)$ into $\isome(\env(Y))$.
\item this is the unique such map for which $\tilde{t}$ is an extension of $t$ for each $t\in \isome(Y)$. \end{enumerate}
Consequently $Y\con \env(X)$ is a linear $g$-embedding. \qed
\end{proposition}

\cor\label{oiu34jt433} 
Suppose that $X$ is \auh, let $Y\con X$ be a subspace, and let $\ga\in \Emb(Y, \env(Y))$. 
Then $\ga$  is such that $\env(\ga Y)=\env(Y)$ if and only if $\ga\in \Emb_\mr{ext}(Y,\env(Y))$, that is, if $\ga$ is a restriction of a surjective isometry of $\env(Y)$. 
In particular,  any partial isometry between full subspaces of $X$ extends to an isometry of $X$.    


\fcor
\prue 
For suppose that $g\in \isome(\env(Y))$, and set $\ga:=g\rest Y$.  Then, by uniqueness, $\bar \ga= g$ and, by Proposition \ref{extisometry},  $\env(\ga Y)=\bar{\ga}\env(Y)=g\env(Y)=\env(Y)$. For suppose now that $\ga\in \Emb(Y,\env(Y))$ is such that $\env(\ga Y)=\env(Y)$. We claim that $\bar \ga\in \isome(\env(Y))$: First of all, $\bar \ga(\env(Y))=\env(\ga Y)\con \env(Y)$, so $\bar\ga\in \Emb(\env(Y),\env(Y))$. Let $\eta\in \Emb(\ga Y, \env(Y))$ be such that $\eta(\ga Y)=Y$ and $\eta \circ \ga =\id_Y$. Then, similarly we get that   $\bar \eta \in \Emb (\env(\ga Y), \env (\ga Y))= \Emb (\env(Y), \env (Y))$. Since $\bar \eta \circ \bar \ga \rest Y= \eta\circ \ga = \id_Y $, it follows by uniqueness of the extensions that $\bar \eta \circ \bar \ga= \id_{\env(Y)}$, and consequently $\bar \eta$ is surjective and $\bar \ga$ as well. 
%
\fprue
Observe that  a subspace $Y$ of $X$  and $g Y$ for $g\in \isome(X)$ are placed similarly inside $X$ because both $Y \equiv gY$ but also the quotient map $\widehat{g}: X/Y \to X/(gY)$, $x+ Y\mapsto g(x)+g Y$ is a surjective isometry compatible with $g\rest Y: Y\to g Y$. In this way, any of the $\isome(X)$-orbits of the action by composition in $\Emb(Y,X)$ is called an {\em isometric position} of $Y$ inside $X$.  In this way, the previous Corollary says that when $X$ is \auh  the isometric position  defined by the inclusion $i: Y\to \env(Y)$ consists of the isometric embeddings $\ga: Y\to \env(Y)$ such that $\env(\ga Y)=\env(Y)$.  From a homological point of view, the exact sequences
$0 \rightarrow Y \rightarrow E(Y) \rightarrow E(Y)/Y \rightarrow 0$
and
$0 \rightarrow \gamma Y \rightarrow E(\gamma Y) \rightarrow E(\gamma Y)/\gamma Y \rightarrow 0$ are isometrically equivalent (in the sense that the extension of $\gamma$ to a map between the envelopes making the diagram commute is an isometric map), and the associated quotient $E(Y)/Y$ is isometrically unique (in the sense that $E(\gamma Y)/\gamma Y$ is independent of $\gamma$).

\begin{definition}[position]
Given a \auh space $X$, the {\em position} of a subspace $Y$ inside its envelope $\env(Y)$ is the $\isome(\env(Y))$-orbit of the inclusion $i : Y\to \env(Y)$. 

The {\em full position} of a subspace $Y$ of $X$ is (if exists) the position of some (any) full isometric copy of $X$. We call {\em full quotient of $X$ by $Y$} the isometrically unique associated quotient $X/Y$.
\end{definition}



 This definition was inspired by a notion of {\em isomorphic position} related to the so-called automorphic space problem, see \cite{CP}.  It is worth mentioning that it does not seem possible to extend the result of uniqueness of full position inside the space $L_1$ to the context of ``almost-isometric" position: It is proved in \cite[Theorem V.1 and Remark (2) on page 284]{GoKaLi} that there is a subspace $X$ of $L_1$ such that (a) every isometric embedding of $X$ into $L_1$ extends uniquely to a surjective isometry on $L_1$, but (b) for every $\de>0$ there is a $\de$-embedding $\ga_\de:X\to L_1$ such that $\inf_{\de>0} d_\mr{BM}(L_1/X, L_1/\ga_\vep(X))>1$, where $d_\mr{BM}$ is the Banach-Mazur (multiplicative) distance. 
 
We finish this subsection with some additional observations on \auh with a stronger extension property.  
Recall that ${\rm Emb}_\de(E,X)$ is the collection of all linear mappings $\ga: E\to X$ such that $(1+\de)^{-1} \nrm{e}\le \nrm{ \ga e}\le (1+\de)\nrm{e}$ for every $e\in  E$.

\begin{definition}[Fraïssé space]\label{fraisse}
A Banach space $X$ is Fra\"iss\'e when for
every $\vep>0$ and every $k \in \N$, there is a $\de>0$ such that 
for every $k$-dimensional subspace $E$ of $X$ and every $t \in {\rm Emb}_\delta(E,X)$, there exists $T \in \isome(X)$ such that
$\|T\rest E-t\| \leq \vep$.
 \end{definition}
 
 The following is the main conjecture of
 \cite{FLMT}.

\begin{conjecture}
The spaces $L_p$ for $p \neq 4,6,..$ and the Gurarij space \Gurarij are the only separable Fra\"iss\'e or even \auh spaces.
\end{conjecture}

It was proved in \cite[Proposition 2.13.]{FLMT} that a Fra\"iss\'e space must contain an isometric copy of the Hilbert space. It remains open whether any \auh space is Fra\"iss\'e, or simply whether it contains an isometric copy of the Hilbert. But even in the \auh case, Hilbertian subspaces seem to be particularly relevant to the study of envelopes.
It should be noted that
we do not know whether every separable Fra\"iss\'e space contains a full copy of the Hilbert space. In Section 3 we shall give a positive answer in the case of the Fra\"iss\'e $L_p$-spaces. For more general Fra\"iss\'e spaces we have the following:

\begin{proposition} Consider $X$ an \auh  Banach space and $Y$ a full Hilbertian subspace of $X$.  Then $Y$ is a minimal full subspace of $X$ and a maximal Hilbertian subspace of $X$. 
\end{proposition}

\begin{proof}
Let $Y$ be some full Hilbertian subspace of $X$. Note that no proper subspace $Z$ of $Y$ has full envelope: Indeed the identity embedding $i_{Z,X}$ of $Z$ into $X$ admits several extensions as an isometric embedding of $Y$ into $X$ and therefore several extensions as an isometric embedding of $X$ into $X$; while the unique isometric embedding of $\env(Z)$ extending $i_{Z,X}$ is $i_{\env(Z),X}$ by Proposition \ref{ext} {\rm (1)}.
On the other hand all superspaces of $Y$ have full envelope. Therefore $Y$ is minimal with full envelope.

This also means that $Y$ is maximal Hilbertian inside $X$. Indeed a non-trivial Hilbertian extension of $Y$ would have full envelope and would not be a minimal full subspace, contradicting the first assertion of the proposition.
\end{proof}

\section{Envelopes in rearrangement invariant spaces and $L_p$-spaces} 

In this section we identify the envelope of  subspaces of $L_p, 1 \leq p <+\infty, p \neq 2,4,6,\ldots$. We start with facts valid for general reflexive spaces, and then for r.i. spaces  on $[0,1]$ (also called symmetric by Peller \cite{Peller}). 
We start with the following result that allows to compute the Korovkin envelope.

\begin{proposition}\label{union} Assume $X$ is reflexive LUR. Let $S$ be a semigroup of contractions on $X$.  Let $Y \subseteq X$ be the closure of a directed sequence $(Y_{i})_{i\in I}$ of subspaces of $X$. Then
$$\env_{S}(Y)=\overline{\bigcup_{i\in I} \env_S(Y_i)}.$$
\end{proposition}

\begin{proof}  Only the direct inclusion is not trivial. For each $i\in I$ let $p_i$  be a  contractive projection  $p_i \in \overline{\mr{conv}(S)}^\mr{WOT}$ onto
 $\env_S(Y_i)$ provided by Proposition \ref{JdLG-kor}. 
Let $p$ be a WOT-cluster point of $\{p_i\}_{i\in I}$, and note that $p \in \overline{\mr{conv}(S)}^\mr{WOT}$ and that $p$ is a contractive projection onto $\overline{\bigcup_{i\in I} \env_S(Y_i)}$. In particular, $p\rest Y=\id_Y$. Therefore by Remark \ref{remrem}, the range of $p$ contains $\env_S(Y)$. 
\end{proof}

\begin{remark}
It does not seem that the previous proposition holds for the algebraic envelope: 
 Assume now that $Y \subseteq X$ is the closure of a directed sequence $(Y_{i})_{i\in I}$ of subspaces of $X$. It is clear that 
$\env_{S}^{\rm alg}(Y) \supseteq \overline{\bigcup_{i\in I} \env_S^{\rm alg}(Y_i)}$. As for the direct inclusion, assume that
$d(x, \overline{\bigcup_{i\in I} \env_S^{\rm alg}(Y_i)})>\epsilon$.  We use the following claim.
\clam
Assume  $X$ 
is reflexive strictly convex. Let $S$ be a semigroup of contractions on $X$ and let $Y \subseteq X$. If $d(x, \env_S^\mr{alg}(Y))>\epsilon$, then there exists $s \in S$ such that $s\rest{Y}=\id_Y$ and $d(x,sx)>\epsilon$.
\fclam
To prove this,  by Proposition \ref{JdLG-alg}, let  $p$ be a contractive projection on 
$\env_S^\mr{alg}(Y)$ belonging to the set
$\overline{{\rm conv}({\rm Stab}_S(Y))}^{\rm WOT}$.
Note that for such $x$ one has that  $\|x-px\| >\epsilon$; therefore $\phi(x-px)>\epsilon$ for some norm $1$ functional $\phi$. Since $p \in \overline{\mr{conv}({\rm Stab}_S(Y))}^\mr{WOT}$, there is some $T \in \mr{conv}({\rm Stab}_S(Y))$ such that $\phi(x-Tx)>\epsilon$ and therefore some $s \in {\rm Stab}_S(Y)$ such that $\phi(x-sx)>\epsilon$.

Once this is established, we can get  $(s_i)_{i \in I}$ with $
s_i$ acting as the identity on $Y$ and $d(x,s_i x)>\epsilon$. Under appropriate hypotheses one would obtain a WOT-cluster point $T$ of  $(s_i)_i$ with
$T\rest{Y}=\id_Y$ and
$d(x,Tx)>\epsilon$. If $T$ were an element of $S$, this would prove that $x$ is not in $\env_{S}^{\rm alg}(Y)$ ; but $T$ only belongs to
$\overline{S}^{\rm WOT}$ a priori. This perhaps explains why the result in the desired equality was peoved for the Korovkin envelope 
In the case of $X=L_p$, $p\neq 2$, we know the following. Suppose that  $Y=L_p([0,1],\Sig)$, and $(\Sig_n)_n$ is an increasing sequence of finite subalgebras whose union is dense in $\Sig$. Then  $\env^\mr{alg}(Y)=\overline{\bigcup_{n}\env^\mr{alg}(L_p([0,1],\Sig_n))}$ only holds when $\Sig$ is a fixed-point subalgebra (see Remark \ref{lknjkio855}), because $\env^\mr{alg}(L_p([0,1],\Sig_n))=L_p([0,1],\Sig_n)$ (Lemma \ref{i8942323}) and consequently the required equality means that $\env^\mr{alg}(Y)=Y$.

\end{remark}

Let us note, using Lancien's theorem about the existence of LUR  renormings preserving the isometry group in separable reflexive Banach spaces \cite{Lancien}, we have the following.

\begin{corollary}\label{unionenveloppe}
Assume $X$ is reflexive and either separable or LUR. Let $Y \subseteq X$ be the closure of a directed sequence $(Y_{i})_{i\in I}$ of subspaces of $X$.  Then
$$\env(Y)=\overline{\bigcup_{i\in I} \env(Y_i)}.$$
\end{corollary}

 We do not know if the above is true for the algebraic envelope.

\subsection{Envelopes in rearrangement invariant spaces}

\

In this subsection we consider rearrangement invariant (r.i.) spaces on $[0,1]$ as defined in \cite[Definition 2.a.1.]{LT2}.  It is well-known that reflexive r.i. spaces on $[0,1]$ are separable (see for example \cite[pp 118-119]{LT2}).

\begin{definition} For $X$ a r.i. space on $I=[0,1]$, and every $\sigma$-subalgebra $\Sigma$ of measurable subsets of $I$, let $X_\Sig$ be the subspace of $X$ consisting of the $\Sigma$-measurable functions in $X$.  \end{definition}

Note that $X_\Sig$ is a unital space.

\begin{lemma}\label{i8942323} If $X$ is a r.i. space on $[0,1]$, and $Y$ is a finite dimensional unital sublattice, then 
$\env^\mr{alg}(Y)=Y$.
\end{lemma}

\begin{proof} 
The sublattice $Y$ is generated by the characteristic functions of $n$ disjoint measurable subsets $(A_i)$ partitioning $[0,1]$.
 Let $f$ be such  that $T(f)=f$ for all $T \in {\rm Stab}_{\isome(X)}(Y)$. 
 We claim that $f\rest A_i$ is constant for each $i$, and consequently, $f\in Y$. We prove the claim by contradiction.  We can find $c \in \R$, and $B, C$ of positive (equal) measure in some $A_i$ such that $f \geq c$  on $B$ and
 $f<c$ on $C$. Let $\sig$ be a measure isomorphism such that $\sig(B)=C$, and $\sig(A_j)=A_j$ for all $j=1,\dots, n$. Observe that the isometry $T_\sig$ determined by $\sig$ belongs to ${\rm Stab}_{\isome(X)}(Y)$, but $T_\sig(f)\neq f$, because $T_\sig(f)\rest C$ is a function with value $\ge c$ while $f\rest C$ has always value $<c$. This is impossible since we are assuming that $f\in \mr{Env}^\mr{alg}(Y)$.
%
\end{proof}

\begin{remark}\label{lknjkio855}
Recall that a $\sig$-subalgebra $\mc B$ of a measure algebra $(\mc A,\mu)$ is called  a {\em fixed-point subalgebra} when there is a set $\Ga$ of measure-preserving isomorphisms of $\mc A$ such that $\mc B= \conj{a\in A}{\pi a=a \text{ for every $\pi\in \Ga$}}$ (see \cite[Section 333]{Frem3}).  This subalgebras are characterized in  \cite[Theorem 333R]{Frem3}, where in particular is shown that fixed-point subalgebras can be determined by a single measure-preserving isomorphism of $(\mc A,\mu)$. It is straightforward to see that 
A unital sublattice $Y$ of $L_p[0,1]$, $p\neq 2$, satisfies that $\env^\mr{alg}(Y)=Y$ exactly when $Y=L_p([0,1],\Sig)$ for some fixed-point  subalgebra of the Borel algebra $\mc B([0,1])$. Moreover, we will see in Theorem  \ref{62} that if $Y\con L_p$, $p\neq 2$, is a unital subspace, then $\env_\mr{min}(Y)=\env(Y)=L_p([0,1],\Sig_Y)$, and    $\env^\mr{alg}(Y)= L_p([0,1],\Sig_Y)$ iff $\Sig_Y$ is a fixed-point subalgebra.
\end{remark}

Note that by   \cite[Theorem 5.25]{Beatasurvey}, this is a $1$-complemented subspace.

\begin{proposition}\label{oi35dsfewfwerwe}  If $X$ is a reflexive  r.i. space on $I=[0,1]$, then for any $\sig$-subalgebra $\Sig$ one has that $\env(X_\Sig)=X_\Sig$. 
\end{proposition}
\begin{proof}  Each $\sig$-algebra $\Sig$ is separable, hence we can find an increasing sequence $(\De_n)_n$ of finite subalgebras of $\Sig$ whose union is dense (see \cite{Ha}), and consequently $\bigcup_nX_{\De_n}$ is dense in $X_\Sig$, because $X$ is a r.i. space.  Since each $X_{\De_n}$ is a finite dimensional unital sublattice, it follows from the previous lemma that $\env^\mr{alg}(X_{\De_n})=X_{\De_n}=\env(X_{\De_n})$.  We know that $X$ is separable, 
so by Corollary \ref{unionenveloppe}  we deduce that $\env(X_{\Sig})=\overline{\bigcup_n \env(X_{\De_n} )}=\overline{\bigcup_n X_{\De_n}}=X_{\Sig}$.
\end{proof}
Note that holds: $X_{\cap_i \Sigma_i}=\bigcap_i X_{\Sigma_i}$. Therefore we may define the following envelope.

\begin{definition}[Conditional envelope] Suppose that $X$ is a r.i. space on $I=[0,1]$. For $Y \subseteq X$,  let $\Sigma_Y$ be the smallest $\sigma$-algebra making all functions of $Y$  measurable. The conditional envelope $\env_\mr{c}(Y)$ is the space $X_{\Sigma_Y}$ of $\Sigma_Y$-measurable functions of $X$.   
\end{definition}
The name conditional we chose is due to the fact that $\env_\mr{c}(Y)=L_p([0,1],\Sig_Y)$ is the range of the conditional expectation projection $\mc E^{\Sig_Y}$.

Note that for reflexive r.i. spaces over $[0,1]$ this defines an envelope which is $1$-complemented.
Observe that $\env_\mr{c}(Y)$ is always unital.  When $X=L_p[0,1]$, $X_\Sig$ is the Lebesgue space $L_p([0,1],\Sig)$, and consequently $\env_\mr{c}(Y)= L_p([0,1],\Sig_Y)$.  We have also the following classical result.

\begin{proposition}\label{oihj34irtjw4jrtw4rt4}
Suppose that $1\le p<\infty$, $p\neq 2$,  and $Y$ is a unital subspace of $L_p[0,1]$. Then,
$$\env_\mr{min}(Y)=\env_\mr{c}(Y)= \mr{Lat}(Y),$$
where $\lat(Y)$ is the lattice generated by $Y$.
%
%
%
\end{proposition}   
\begin{proof}
We  first prove that  $\env_\mr{c}(Y)=\env_\mr{min}(Y)$. We know that  conditional expectation operator $\mc E^{\Sig_Y}$ is a contractive projection whose range is $L_p([0,1],\Sig_Y)= \env_\mr{c}(Y)$. 
On the other hand, 
  classical results by Douglas \cite{Doug} (for $p=1$) and by Ando \cite{Ando} (for the rest of $p$'s) states that a contractive projection that fixes $\mathbbm 1_{[0,1]}$ must be  a conditional expectation defined by some $\sig$-subalgebra. 
  Hence any contractive projection $q$ whose range $R(q)$ contains $Y$ satisfies
  $R(q)=L_p([0,1],\Sig)$ for some $\sig$-subalgebra $\Sig$.  Since $Y\con R(q)$, it follows that $\Sig_Y\con \Sig$,  so $\env_\mr{c}(Y)=L_p([0,1],\Sig_Y)\con L_p([0,1],\Sig)=R(q)$. This proves that $\env_{\rm c}(Y)=\env_{\rm min}(Y)$ (including that the latter exists when $p=1$).
  
  
We now prove  that $\env_\mr{c}(Y)= \mr{Lat}(Y)$. We start with the following.
 \begin{claim} \label{njiwejriojewoirew} Suppose that $Z$ is a unital sublattice. Then
 $\Sig_Z= \conj{A}{\mathbbm 1_A\in Z}$ and $Z= L_p([0,1], \Sig_Z)$.
 
 \end{claim}   
\prucl
Set $\De:=\conj{A}{\mathbbm 1_A\in Z}$. Obviously, $\De\con \Sig_Z$.   
 Observe that in general for an arbitrary subspace $V$ one has that $\Sig_V$ is the $\sig$-subalgebra generated by the sets $\{f>c\}$ for $c>0$ and $f\in V$, and when $V$ is assumed to be lattice then those $f\in V$ can be assumed to be positive. Fix now a positive $f\in Z$, $c>0$, and set $g:= (f- c\mathbbm 1_{[0,1]})^+\in Z$. Then, 
$$\mathbbm 1_{\{f>c\}}=\mathbbm 1_{\{g>0\}}= (\sup_{n\in \N} n \cdot g )\wedge \mathbbm 1_{[0,1]}\in Z,$$
and consequently $\{f>c\}\in  \De$. By the previous remark, $\Sig_Z\con \De$.  Finally, trivially we have that $Z\con L_p([0,1], \Sig_Z)$; as for the reverse inclusion, observe that it follows from  the equality $\Sig_Z=\conj{A}{\mathbbm 1_A\in Z}$  that every simple function  in $ L_p([0,1], \Sig_Z)$ belongs to $Z$, hence  $L_p([0,1],\Sig_Z)\con Z$. 
\fprucl	
We have by minimality that $\mr{Lat}(Y)\con L_p([0,1],\Sig_Y)=\env_\mr{c}(Y)$. Since $\mr{Lat}(Y)$ is 1-complemented, it follows that $\env_\mr{c}(Y)=\env_\mr{min}(Y)\con \mr{Lat}(Y)$ and we are done.
\end{proof}


\begin{lemma}\label{inclu} If $X$ is a reflexive r.i. space on $[0,1]$,  then for any  unital subspace $Y$ one has that  
$$\env_\mr{min}(Y) \subseteq \env(Y) \subseteq \env_\mr{c}(Y).$$ 
 \end{lemma}
 \begin{proof}
From Proposition \ref{ghfddddd} and Proposition  \ref{uhiu4iur4rt4} we know that $\env(Y)$ is 1-complemented, hence $\env_\mr{min}(Y)\con \env(Y)$. Also, it follows from Proposition \ref{oi35dsfewfwerwe} that  $\env(\env_\mr{c}(Y))=\env_\mr{c}(Y)$ and consequently,  $\env(Y) \subseteq \env_\mr{c}(Y)$.
\end{proof}


\begin{lemma}\label{ghghgh}
If $Y$ is a unital subspace of $L_1$, then 
$\lat(Y) \subseteq \env(Y)$.  
\end{lemma}

\begin{proof}By the equimeasurability formula, any isometric map $t$ from  $Y$ into $L_1$ extends uniquely to an isometric map $T$ on $\lat(Y)$   by the formula $$T(|f|):=\frac{t(\mathbbm 1)}{|t(\mathbbm 1)|}|t(f)|.$$ 
So, fix $f\in Y$, and  assume $(t_n)_n$ tends SOT to $t$ on $Y$.  We claim that $(t_n(|f|))_n$ converges to $T(|f|)$,  that by the
  characterization of $\env(Y)$ will mean that $|f| \in \env(Y)$.
Let us prove this claim. 
By composing by an isometry acting by change of signs, it is enough to prove the claim when $t(\mathbbm 1)$ is non-negative, and by composing with an isometric embedding of $L_1(\supp(t(\mathbbm 1)))$ onto $L_1$ we may assume that  $t(\mathbbm 1)$ is positive, and therefore  that $T(|f|)=|t(f)|$.
We compute:
\begin{equation}\label{lkwnjriweljrnewklrew}
T_n(|f|)-T(|f|)=\frac{t_n (\mathbbm 1)}{|t_n (\mathbbm 1)|} |t_n (f)| -  |t (f)|
= \frac{t_n (\mathbbm 1)}{|t_n (\mathbbm 1)|} (|t_n (f)|-|t(f)|)+(\frac{t_n (\mathbbm 1)}{|t_n (\mathbbm 1)|}  - \mathbbm 1) |t (f)|    
\end{equation}
The norm of the first part of the sum is  at most
$\int ||t_n (f)|-|t(f)|| \leq \int |(t_n-t)(f)|$, and this tends to $0$.  Since $(t_n(\mathbbm 1))_n$ tends to $t(\mathbbm 1)$ in the $L_1$-norm,   the measure of the sets $A_n:=\{s\in [0,1]\, : \, t_n(1)(s)\le 0\}$  tend to $0$.
The norm  of the second summand in \eqref{lkwnjriweljrnewklrew} is  controlled by
$2 \int_{A_n}|t(f)|$ and therefore tends to $0$ as well.
\end{proof}


\begin{theorem}\label{62} If $1 \leq p<+\infty, p \neq 2$ and $Y$ is a unital subspace of $L_p$ then  
$$\env(Y)=\lat(Y)=\env_\mr{c}(Y)=\env_\mr{min}(Y).$$ \end{theorem}
\begin{proof} 
For  $p>1$,  we know from  Proposition \ref{oihj34irtjw4jrtw4rt4} and Lemma \ref{inclu}  that 
$$\lat(Y)=\env_\mr{min}(Y)\con \env(Y)\con \env_\mr{c}(Y)=\lat(Y).$$
The proof for $p=1$ uses the Mazur maps.
For some fixed $1<p \neq 2$ we let $\phi: S_{L_1} \rightarrow S_{L_p}$ denote the Mazur map, extend it by homogeneity to a map between $L_1$ and $L_p$. Observe that $T \in \isome(L_1)$ if and only if $\phi T \phi^{-1} \in \isome(L_p)$. Let us also note that if $Y$ is a unital lattice of $L_1$ then $\phi(Y)=L_p \cap Y$ which is a unital sublattice of $L_p$. 
As we have seen in Claim \ref{njiwejriojewoirew}, there is a 1-1 correspondence between the class of unital sublattices of $L_p$ and the class of sub-$\sigma$-algebras of Borel subsets of $[0,1]$,  $$Y \mapsto \Sigma_Y=\{A \in {\mathcal B}: 1_A \in Y\}$$ and
$$\Sigma \mapsto Y_\Sigma=L_p([0,1],\Sigma).$$
%
To avoid misinterpretations, we use   the terminology $\env_{(L_p)}(Z)$ to denote the isometric envelope of some $Z\con L_p[0,1]$ in $L_p[0,1]$.  
\begin{claim}
$\phi(\env_{(L_1)}(Y)\con \env_{L_p}(\phi(Y))$. 
\end{claim}  
Admitting the claim, we deduce that if $Z\con L_1[0,1]$ is a unital sublattice, then so is $\phi(Z)$, hence $\phi(\env_{(L_1)}(Z))\con \env_{(L_p)}(\phi(Z))=\phi(Z)$, and consequently $\env_{(L_1)}(Z)=Z$, because $\phi$ is 1-1.  This implies that  
 $\env(Y) \subseteq \env(\lat(Y)) = \lat(Y)$,
and the equality $\env(Y)=\lat(Y)$ follows from Lemma \ref{ghghgh}.  The rest of the equalities are in Theorem \ref{oihj34irtjw4jrtw4rt4}. 

To prove the claim, assume $(T_i)_i$ is a net of isometries of $L_p$ so that
$(T_i (g))_i$ tends to $g$ for all $g \in \phi(Y)$. This is equivalent to that $(\phi U_i (y))_i$  tends to $\phi (y)$ for all $y$ in $Y$ where each $U_i:=\phi^{-1} T_i \phi$ is the associated   isometry on $L_1$ to $T_i$. Since the Mazur map is an homeomorphism (see \cite{Mazur}), this is equivalent to that $(U_i)_i$ converges to $\id$ pointwise on $Y$, which, since $f \in \env_{(L_1)}(Y)$, implies that $(U_i (f))_i$ tends to $f$. Therefore $(T_i \phi (f))_i$ converges to $\phi f$, because $T_i \phi (f)=\phi U_i (f)$.
\end{proof}

\begin{remark} \label{oi32rio23rfe}
 It follows easily that the previous proposition  is true for the Lebesgue spaces $L_p(A)$ for every Borel subset $A$ of the unit interval. Hence, if $Y\con L_p[0,1]$, $p\neq 2$,  is a subspace such that $\mathbbm 1_A\in Y$ where $A$ is the  support of $Y$,  then 
 $\env(Y)=\lat(Y)=\env_\mr{c}(Y)=\env_\mr{min}(Y).$ This is so, because all these envelopes on the space $L_p(A)$ are equal to their versions in the full space $L_p[0,1]$. 
 \end{remark}  	

\begin{remark}
In the case $1<p<+\infty$ this result can also be deduced from Proposition \ref{ghfddddd}(b) and the description by Peller of the WOT closure ${\mathcal M}$ of the isometry group of $L_p$ \cite{Peller}, from which it follows that any contractive projection on $L_p$ belongs to ${\mathcal M}$.
There does not seem to be such a direct proof in the case $p=1$, since our WOT characterization of the envelope is, apparently, only valid in the reflexive case. 
\end{remark}
 
We have the following (partial) extension of Theorem \ref{62}.
\begin{proposition}  Suppose that $1 \leq p<+\infty$, and let $Y$ be a subspace of $L_p$. Then we have that
\begin{enumerate}[\rm(1)]
 \item  $\env(Y)=\env_{\rm min}(Y)$.
 \item 
If 
$p$ is not even, and $g\in Y$ is of full support in $Y$ then
$$\env(Y)= g \cdot \env(Y/g)=g \cdot \lat(Y/g)= g\cdot \env_\mr{c}(Y/g).$$

\end{enumerate}
 
\end{proposition}

\begin{proof}
Pick $g$ of full support in $Y$. To prove (1) first assume $Y$ has full support. Let
$T$ be an isometry on $L_p$ sending $g$ to $1$, so that $TY$ is unital.
Then (1) follows from the equality $\env(TY)=\env_{\rm min}(TY)$ (Proposition \ref{62}), and  from Remarks
\ref{isomenvmin} and before Example \ref{hilbert} about preservation of the minimal and isometric envelopes by isometries. When $Y$ does not have full support the proof is similar using Remark 
\ref{oi32rio23rfe}.

The last two equalities in (2) for {\em every $p\neq 2$} also follow from Remark \ref{oi32rio23rfe}. We have to show that $\env(Y)= g \env(Y/g)$ for non-even $p$'s. 
It is well-known that for every $1\le p<\infty$, the envelope $\env_\mr{c}(Y)=L_p([0,1],\Sig_Y)$ is the closure of the subspace of functions $B(f_1,\dots,f_n)$ where $f_1,\dots,f_n\in Y$ and $B: \R^n\to \R$ is Borel.  Suppose now that $p\notin 2\N$, fix $Y\con L_p[0,1]$ and $g\in Y$ with full support.  
 Let $t$ be an isometry between a unital subspace $Z$ of $L_p$ and $Y$ sending $\mathbbm 1$ to $g$. Since $L_p[0,1]$ is an approximately ultrahomogeneous space, it follows that there is  a unique extension  $T$  of $t$ from $\env(Z)$ onto $\env(Y)$.  In fact, $T$ is defined by 
 $$T(B(f_1,\dots, f_n))= g \cdot B(tf_1/g,\dots, t f_n /g)$$
 for every $f_1,\dots, f_n\in Z$ and every Borel function $B:\R^n\to \R$ (see \cite[Theorem 4]{koldo}).  Hence, 
 \begin{align*}
 \env(Y)= &T(\env(Z))= \conj{T(B(f_1,\dots, f_n))}{f_1,\dots, f_n\in Z, B\text{ Borel }}=\\
 =& 
 g\cdot\conj{B(g_1/g,\dots g_n/g)}{g_1\dots, g_n\in Y, B\text{ Borel}}= g \cdot L_p(\supp g, \Sig_{Y/g})= \\
 =& g \cdot L_p([0,1],\Sig_{Y/g})=g \cdot \env_\mr{c}(Y/g)=g \cdot \env(Y/g).
 \end{align*}
%
%
%
%
\end{proof}

%
%
%

Finally we may observe the following, as a consequence of the description by Peller of the WOT-closure of the isometry group in separable r.i. spaces on $[0,1]$ different from the $L_p$'s. Recall that a contractive map on a r.i. space on $[0,1]$ is absolute if it defines a contraction on $L_1$ and on $L_\infty$.

\begin{theorem}   The $L_p$'s, $1 < p<+\infty$ are the unique reflexive r.i. spaces on $[0,1]$ for which all $1$-complemented subspaces ($1$-dimensional subspaces is enough) are envelopes. 
\end{theorem}

\begin{proof}  Let $X$ be a reflexive   r.i. space on $[0,1]$. In particular, $X$ is separable.  The assertion that all $1$-dimensional subspaces are envelopes implies, by  Proposition \ref{uhiu4iur4rt4}, that any $1$-dimensional subspace is the range of a contractive projection belonging to the WOT-closure of the convex hull of the isometry group. On the other hand,  it is a classical result (see \cite[Theorem 2]{Zaidenberg}) that if  $X$ is different from one of the $L_p$'s, then it follows that 
every isometry of $X$ is absolute. Since absolute contractions are closed under convex combinations and WOT-limits, it follows that under our hypothesis every 1-dimensional subspace of $X$ is 1-complemented by an absolute contractive projection, and we are going to see (probably a well-known result) in the next claim that this is impossible. 


\clam
Every reflexive r.i. space on $[0,1]$ admits a $1$-dimensional subspace which is not the range of an absolute projection.
\fclam
\prucl
Assume otherwise.  We are going to see that if  $u$ and $v$ are disjointly supported elements of $L_\infty$ and $\|u\|_\infty>\|v\|_\infty$, then $\|u+v\|=\|u\|$. Since $\|v\| \leq \|u+v\|$ this implies that actually
$\|u+v\|=\max(\|u\|, \|v\|)$; by induction we would find an isometric copy of $\ell_\infty$ in $X$,  which is impossible by hypothesis, contradiction. Let us see that $\nrm{u+v}=\nrm{u}$. We may assume without loss of generality that $x_0:=u+v$ has norm $1$ in $X$. Let
$P$ define an absolute projection with range equal to the span of $x_0$.
Since $P$ is a rank $1$ projection on $X$, it must satisfy the formula
$P(x)=\langle\phi_0,x\rangle x_0$ where $\phi_0 \in X^*$ and
 $\phi_0(x_0)=\nrm{\phi_0}=1$. Since $P$ is also a contraction in $L_\infty$, let us define $\ga: L_\infty \to \R$ by $\ga(x):= \nrm{x}_\infty-\langle \phi_0, x\rangle \nrm{x_0}_\infty $, and let us observe that 
$$\ga(x) = \nrm{x}_\infty \pm  \nrm{P(x)}_\infty \ge 0.$$
We see now that $\phi_0$ and $v$ are disjointly supported: Let $h:=\mathbbm 1_A$ where $A \subseteq \supp v$. Note that for $|t|$ small enough one has that $\nrm{x_0+th}_\infty=\nrm{x_0}_\infty$, hence 
$$\ga(x_0+th)=\nrm{x_0+th}_\infty-(1+t\langle\phi_0,h\rangle)\nrm{x_0}_\infty=-t\nrm{x_0}_\infty \langle\phi_0,h\rangle.$$ 
Since $\gamma$ is non-negative we deduce that $\phi_0(h)=0$, and this implies that $\supp(\phi_0) \cap \supp(v) =\emptyset$ and therefore
$$\phi_0(u)=\phi_0(x_0),
$$ so $\phi_0$ also norms $u$ and $\|x_0\|=\|u\|$.
\fprucl
\end{proof}

We do not know whether the minimal and  isometric envelopes coincide for an arbitrary reflexive Fraïssé or even \auh Banach space. Going further, it would be interesting to study the relationship between the envelopes in non-reflexive Fraïssé spaces, for example in the Gurarij space, or even in quasi-Banach spaces, such as  $L_p[0,1]$ for $0<p<1$.



\subsection{Full copies of $L_q$-spaces inside $L_p$}
We study the envelopes of $L_q$-spaces inside $L_p$, obtaining the following.

\begin{theorem}\label{3p4jroijrioewrje} Assume that $1 \leq p<+\infty$ is not even.
\begin{enumerate}[\rm(1)]
\item the envelope of any Hilbertian subspace of $L_p$ of dimension at least $2$ is isometric to $L_p$.

\item
If $q$ satisfies $1 \leq p \leq q \leq 2$, then the envelopes of $L_q$ and $\ell_q$  inside $L_p$ are isometric to $L_p$.
\end{enumerate} 
Consequently, 
\begin{enumerate}[\rm(1)]\addtocounter{enumi}{2}
    
    \item  $L_p$ admits  full copies of $\ell_2^n$, $n\geq 2$,  $\ell_2$,  and of  $L_q$ and  $\ell_q$ provided that $1 \leq p \leq q \leq 2$.

\end{enumerate}
 
\end{theorem}

\begin{proof} In fact, we have the following formally stronger statement: Assume $1 \leq p <+\infty$, $p$ not even and let $Y$ be a subspace of $L_p$ such that one of the two conditions hold:
\begin{enumerate}[i)]
\item the group $\isome(Y)$ acts almost transitively on $S_Y$ and $Y$ does not embed isometrically into $\ell_p$;
\item $Y$ has a $1$-symmetric Schauder basis $(e_n)$ such that
$\inf_n \frac{\| \sum_{i=1}^n e_i\|_Y}{n^{1/p}}=0$
\end{enumerate}
Then the envelope of $Y$ is isometric to $L_p$. In particular some isometric copy of $Y$ has full envelope.

Its proof goes as follows.  
Composing with an isometric embedding of $Y$ into $L_p$ sending some vector of full support in $Y$ to $1$, we may assume $Y$ is unital, hence $\env(Y)=L_p([0,1],\Sig_Y)$. The proof is finished once we argue that $\Sig_Y$ does not have atoms. Suppose otherwise. Then $\env(Y)= L_p(A_1,\Sig_1)\oplus_p L_p[A_2,\Sig_2)$ where $\Sig_1, \Sig_2\con \Sig_Y$   are atomless and atomic $\sig$-algebras in $A_1$ and $A_2$, respectively, $A_1\sqcup A_2=[0,1]$, and $\Sig_2$ is non-empty. 

 We denote by $i$ the inclusion map of $Y$ into $L_p(A_1,\Sig_1) \oplus_p L_p(A_2,\Sig_2)$, and 
choosing $i_1=0$ if necessary, we write $i=(i_1,i_2)$ with respect to the decomposition $L_p(A_1,\Sig_1) \oplus_p L_p(A_2,\Sig_2)$, that is $i_j(f)= f\rest A_j$ for every $f\in Y$. Note  that by the Banach-Lamperti formula for isometries of $L_p$-spaces, all isometries on this space can be written as $T=(T_1,T_2)$ where each $T_1$  is an isometry on $L_p(A_i,\Sig_i)$ for $i=1,2$.  Observe that $i_2\neq 0$, because otherwise $A_2=\emptyset$.
\clam
The norm of $i_2 y$ is constant on   $\isome(Y)$-orbits. 
\fclam
\prucl
For suppose that $t\in \isome(Y)$ and $y\in Y$. Since $p$ is not even, we know that $L_p$ is \auh, and it follows from Proposition \ref{extisometry} that there is unique extension of $t$ to $T\in \isome(\env(Y))$. Then  $(i_1 t y, i_2  ty)=ity = Tiy=(T_1 i_1 y, T_2 i_2 y)$,  and consequently $\nrm{i_2 t y}=\nrm{T_2 i_2 y}=\nrm{i_2 y}$.
\fprucl
 Suppose that {\rm i)} holds.   Then it follows from the previous claim that the norm of $i_2$ has constant value $k>0$ on $S_Y$ (because $i_2\neq 0$).  Hence,  $i_2/k$ is an isometric embedding of $Y$ into $L_p[A_2,\Sig_2)$, and since $\Sig_2$ is atomic, $L_p(A_2,\Sig_2)$ is isometric to $\ell_p(I)$, where $I$ is the (countable) set of atoms of $\Sig_2$, contradicting the hypothesis. 
 
 Suppose that {\rm ii)} holds.  Since $(e_n)_n$ is a 1-symmetric basis of $Y$, the isometries of $Y$ act transitively on the basis, and by the previous claim this implies that the norm of $i_2 e_n$ for $n \in \N$ is constant, with value $c$. As before, we have that $c>0$.  By $w^*$-compactness there exists $u \in \ell_p^{**}$ and $N$  an infinite subset of $\N$ such that $(i_2 e_n)_{n \in N} \rightarrow^{w^*} u$. 
We claim that $u=0$.
To see this, note that for any normalized $\phi$ in $\ell_p^*$, and any $n \in \N$ we can find a finite $F \subseteq N$ of cardinality $n$ such that
$\phi(\sum_{j \in F} i_2 e_j) \geq \frac{1}{2} n\phi(u)$. It follows
that
$n \phi(u) \leq 2 \|i_2(\sum_{j \in F} e_i)\| \leq 2\|i_2\| \|\sum_{j=1}^n e_j\|_Y$. The estimate in {\rm ii)} then implies that $\phi(u)=0$ and therefore the claim that $u=0$.
Now since $(i_2 e_n)_{n \in N}$ tends weakly to $0$ in $\ell_p$, we can, by a gliding hump argument and passing to a subset of $N$, assume that the sequence $(i_2 e_n)_{n \in N}$ is almost sucessive in the sense that
$\|\sum_{j \in F} i_2 e_j\| \geq \frac{1}{2}(\sum_{j \in F}\|i_2 e_j\|^p)^{1/p}$ whenever $F \subseteq N$ is finite.
Therefore
$c|F|^{1/p} \leq 2 \|i_2\| \|\sum_{j \in F} e_j\|_Y$, contradicting ii) when $|F|$ is large enough.
\end{proof}

\subsection{On $g$-embeddings of $\ell_2$ into $L_p$ and the spaces $L_p/\ell_2$ }
Knowing that $\ell_2$ admits a full copy in $L_p$, we have an associated exact sequence
$$0 \rightarrow \ell_2 \rightarrow L_p \rightarrow L_p/\ell_2 \rightarrow 0,$$ to which
we may apply  Corollary \ref{oiu34jt433} and the  commentary thereafter. 
\begin{definition}
[The full quotient space $L_p/\ell_2$] Let  $1\le p<\infty$, $p\notin 2\N$.  The quotient $L_p/\ell_2$ of $L_p$ by any full copy of $\ell_2$ is isometrically unique. We call this space the full quotient of $L_p$ by $\ell_2$.\end{definition}

\begin{proposition}
For each $1\le p<\infty$, $p\notin 2\N +4$, there is a linear $g$-embedding (see definition just before Definition \ref{extisometry}) of $U(\ell_2)$ into $\isome(L_p)$. It is defined by extension of unitaries on any fixed full copy of $\ell_2$ inside $L_p$, to isometries on $L_p$.
\end{proposition}
\begin{proof}
This is a  consequence of Theorem \ref{3p4jroijrioewrje} and of Proposition \ref{extisometry}. 
\end{proof}

Many  questions  seem to remain open about concrete representations of full copies of $\ell_2$ inside $L_p$, about the exact sequence
$$0 \rightarrow \ell_2 \rightarrow L_p \rightarrow L_p/\ell_2 \rightarrow 0$$
associated to the linear $g$-embedding of $\ell_2$ inside $L_p$, for appropriate values of $p$,  about the associated full quotient of $L_p$ by $\ell_2$, or about  the topological embedding of $U(H)$
 into $\isome(L_p)$ from a concrete point of view,  . We shall not  address them in this paper but we do make a few initial comments below.

If $1<p<+\infty$, then the exact sequence above splits, by complementation of $\ell_2$ inside $L_p$. This means that $L_p/\ell_2$ is isomorphic  to $L_p$ (since it is isomorphic to a complemented subspace of $L_p$, and since $L_p$ is primary). 
On the other hand when $p=1$, then
$$0 \rightarrow \ell_2 \rightarrow L_1 \rightarrow L_1/\ell_2 \rightarrow 0$$
does not split, and the isomorphic type of the full quotient $L_1/\ell_2$ seems to be unknown.

Many other homological questions arise:
 what description can be made of the induced action of $U(\ell_2)$ on
$L_p/\ell_2$ (which is actually an action by bounded isomorphisms on $L_p$ when $p>1$)?
   for $1<p<+\infty$, does $\ell_2$ admit an $U(\ell_2)$-invariant summand inside $L_p$? if it does not, then what can be said of the quasilinear map $\Omega$ associated to the exact sequence and of the commutator estimates of 
$\Omega$ with the action of the unitary group on $L_p$? How may the dual exact sequence  be described? For details we refer to \cite{CG}, for the general theory of exact sequences of Banach spaces, and to \cite{CF}, for a theory of their compatibility with group actions.

Another direction which seems worthwhile to explore is the relation between the isometric envelopes of the Hilbert space in an $L_p$-space and Gaussian Hilbert spaces. See in particular the extension Theorem 4.12 in
\cite{Jansson}. Note however that the theory of Gaussian Hilbert spaces seems to be limited to the context of $L_p$-spaces and does not seem to generalize to the \auh or Fra\"iss\'e situation.



\section{Complementations in \auh spaces: numerical estimates}

The aim of this section is to establish relations between the constants of complementation of euclidean or hilbertian copies, the type and cotype of the space, and the AUH or Fra\"iss\'e property. Recall that a Fra\"iss\'e Banach space $X$ always contain isometric copies of the Hilbert space. We shall see that the constant of complementation of such copies in $X$ determines for which values of $p$ the space $X$ is possibly isometric to $L_p$. We shall also relate these values to the type and cotype of the space.

In consequence we shall obtain ``local versions" of the fact that $L_p$-spaces are Fra\"iss\'e only if $p \notin 2\N+4$, and reinforce the conjecture that those spaces, together with the Gurarij space, are the only separable Fra\"iss\'e spaces. 

\subsection{Complementations in \auh spaces}
Our first result relates to constants of complementation in \auh-spaces. This extends some results known for $L_p$-spaces, but our point is to obtain these only through properties of the isometry group. 

\begin{proposition}\label{prpeprpep} 
Let $X$ be \auh and $1$-complemented in its bidual. Then any isometric embedding of a subspace $Y$ of $X$ into $X$ extends to a contraction on $X$.
\end{proposition}

\begin{proof}  Let $t$ be an embedding of a subspace $Y$ of $X$.
For any $F$ finite-dimensional subspace of $Y$ and $\varepsilon>0$, there is a map $T_{F,\varepsilon} \in \isome(X)$ such that
$\|(t-T_{F,\varepsilon}){\rest F}\| \leq \varepsilon$. Using a non-trivial ultrafilter refining the natural filter on the set of
$(F,\varepsilon)$, we define a map $T: X \rightarrow X^{**}$
by $Tx=w^*-\lim_U T_{F,\varepsilon}x$. If $P$ is a norm $1$ projection
from $X^{**}$ onto $X$ then $PT$ defines a contraction on $X$ extending $t$.
\end{proof}

\begin{proposition}\label{newproj}  Let $X$ be \auh, and let
$Y$ be $C$-complemented in $X$ for some $C \geq 1$. Then  any isometric copy of $Y$ inside $X$ is also $C$-complemented, as soon as one of the following conditions holds:
\begin{enumerate}[\rm(1)]
\item $X$ is $1$-complemented in its bidual, or
\item $Y$ is finite dimensional, or
\item  $Y$ is $1$-complemented in its bidual and can be written as the closure of an increasing union of a net $(F_i)_i$ of finite dimensional subspaces such that $F_i$ is $1$-complemented in $Y$ for every $i$.
\end{enumerate}
\end{proposition}

\begin{proof} Let $Y$ be $C$-complemented by a projection $p$ and $t$ be an isometric embedding of $Y$ into $X$.
(1): Using Proposition \ref{prpeprpep}, $tY$ is complemented by $tpT'$, if $T'$ is a contraction on $X$ extending $t^{-1}$. (2): for any $\varepsilon>0$ we can find, by classical perturbation arguments, a surjective $1+\epsilon$-isometry $T_\varepsilon$ on $X$ extending $t$, and then 
$p_\varepsilon:=T_\varepsilon p T_\varepsilon^{-1}$ is a projection onto $tY$ of norm at most $C(1+\varepsilon)^2$. A compactness argument provides a projection onto $tY$ of norm at most $C$.
(3): Applying (2) we find a projection $p_i$ onto $tF_i$ of norm at most $C$ for each $i$. Using a non-trivial ultrafilter refining the net, we define a map $p: X \rightarrow (tY)^{**}$ by
$px=w^*-\lim_U p_i x$. If $P$ is a norm $1$ projection from $(tY)^{**}$ onto $tY$ then $Pp$ defines a projection onto $tY$ of norm at most $C$.
\end{proof}

Note that as a consequence of this and of the \auh properties of $L_p$-spaces, we recover a known result in the theory of $L_p$-spaces (\cite{ran02}), i.e. 
the above holds inside $L_p(0,1), 1 \leq p <+\infty, p \notin 2\N+4$.
It also proved in \cite{ran02} that the result does not hold for $p \in 2\N+4$, since in this case there is  a complemented subspace admitting an isometric uncomplemented copy. One may actually prove the following:

\begin{remark}\label{complu} No \auh space
can contain a complemented isometric copy of any $L_p$ with $p \in 2\N+4$. This follows from the existence of a constant $K$ and a sequence of pairwise isometric finite dimensional subspaces $F_n$ and $G_n$ of $L_p$ such that
$F_n$ is $K$-complementend and $G_n$ is not $n$-complemented in $L_p$, for each $n \in \N$. It is easy to see that this fact, together with item (2) of Proposition \ref{newproj}, is incompatible with complementation of a copy of $L_p$ inside $X$. The existence of $K, (F_n)$ and $(G_n)$ is based on the unconditionality of the subspaces used in \cite{ran02} and is detailed in the proof of Proposition 2.10 of \cite{FLMT}.  \end{remark}

\

 As for the Hilbert space $H$, a  classical result states that $U(H)$ is WOT-dense in ${\mathcal L}_1(H)$ (see for example \cite{Peller}).  This and  Proposition \ref{prpeprpep} gives us the following.
 
\begin{corollary} \label{oij3riojweiojeoijwer}   Assume $X$ is  \auh $1$-complemented in its bidual. Then every operator defined on  a hilbertian subspace of $X$ extends to an operator of same norm on $X$.
\end{corollary}
\begin{proof} For suppose that $Y\con X$ is hilbertian and $t:Y\to Y$ be a contraction of norm 1. We choose a net $(t_i)_i$ of isometries of $Y$ converging WOT to $T$. By Proposition \ref{prpeprpep}, we can extend each $t_i$ to norm one operators $T_i:X\to X$.  Let $U:X\to X^{**}$ be WOT-limit of $(T_i)_i$, and let $P:X^{**}\to X$ be a contractive projection. Then $P U$ has norm 1 and extends $t$.  
\end{proof}
Note that this holds in particular for
any Fra\"iss\'e space $1$-complemented in its bidual, and therefore for all the separable Fra\"iss\'e $L_p$-spaces.

A similar result seems possible  for spaces $X$  for which the unital algebra ${\mathcal L}(X)$ is {\em unitary}, that is,  when  the convex hull of its isometries is norm dense in ${\mathcal L}_1(X)$. 
 It is classical that ${\mathcal L}(H)$ is unitary in the complex case and the real case is due to Navarro-Pascual and Navarro \cite{NN} who actually prove that the infinite convex hull is equal to ${\mathcal L}_1(H)$.


\begin{definition}
For $X$ a Banach space, $1 \leq p \leq \infty$, and $n \in \N$, let us denote
$c_p^n(X) \in [1,+\infty]$  the infimum of  constants of complementation of an isometric copy of $\ell_p^n$ inside $X$,
and $c_p(X)$ the infimum  constants of complementation of an isometric copy of $L_p$ inside $X$. Let us also denote by $\fr(X)$ the set of $p$'s such that $\ell_p$ is finitely representable in $X$.  
\end{definition}

\begin{proposition}\label{complecomple}
Assume $X$ is Fra\"iss\'e. The following holds for every $p \in \fr(X), 1 \leq p<+\infty$:
\begin{enumerate}[\rm(1)]
\item  $c_p^n(X)$ is attained.
\item $c_p(X)$ is attained when it is finite.
\item We have $c_p(X)=\lim_n c_p^n(X)$.
\item If there is $C \geq 1$ such that for any $n$, and for any $\varepsilon>0$, $X$ contains a $C$-complemented $1+\epsilon$-copy of $\ell_p^n$, then $c_p(X) \leq C$.
\end{enumerate}
Furthermore
\begin{enumerate}[\rm(1)]\addtocounter{enumi}{4}
\item for fixed $n$, the map $p \mapsto c_p^n(X)$ is continuous on $\fr(X)$.      
\end{enumerate}

\end{proposition}

\begin{proof}  Recall that $L_p$ embeds into $X$ as soon as $p \in \fr(X)$ (see \cite{FLMT}). Here, except for (1), we are using that the space $X$ is assumed to be Fraïssé and not only \auh (although we do not know examples of \auh non Fraïssé spaces). 
(1): Since $X$ is \auh, any isometric map between two copies of $\ell_p^n$ may be approximated by a global isometry, therefore extends to a $1+\varepsilon$-isometry for $\varepsilon$ arbitrarily small, and we deduce that any isometric copy of $\ell_p^n$ is complemented by a projection of norm at most $c_p^n(X)+\delta$, for arbitrary $\delta>0$. We may then define a projection as some weak limit of almost optimal projections along an ultrafilter to see that all copies of $\ell_p^n$ are complemented by
a projection of norm $c_p^n(X)$.

 Note that the sequence $c_p^n(X)$ is non-decreasing, and assume it is bounded by $C$. 
Writing $L_p$ as the closure of the increasing union of $\ell_p^n$'s, and again defining a projection onto $L_p$ as some weak limit of projections on the $\ell_p^n$'s (using that $L_p$  is $1$-complemented in its bidual),  we deduce that both (2) and (3) hold.

(4): Let $\varepsilon>0$, $\eta>0$, and consider a $C$-complemented $1+\eta$-isometric copy $F$ of $\ell_p^n$. If $\eta$ had been chosen small enough, the Fra\"iss\'e property allows to extend a partial $1+\eta$-isometry $t$ between $F$ and an isometric copy $G$ of $\ell_p^n$  to a global isomorphism $T$ on $X$ such that
$\|T\| \|T^{-1}\| <1+\varepsilon$, guaranteeing that $G$ is $C(1+\varepsilon)$-complemented. Then we deduce that $c_p^n(X) \leq C$, and, by c), that $c_p(X) \leq C$.

(5): Let $p \in \fr(X),$ $\varepsilon>0$, $\eta>0$. For $q$ close enough to $p$ in $\fr(X)$, there is a partial $1+\eta$-isometry $t$ between copies of $\ell_p^n$
and $\ell_q^n$. If $\eta$ had been chosen small enough, the Fra\"iss\'e property allows to extend $t$ to a global isomorphism $T$ on $X$ such that
$\|T\| \|T^{-1}\| <1+\varepsilon$, guaranteeing that
$c_q^n(X) \leq (1+\varepsilon)c_p^n(X)$.
\end{proof}

In the case (4) we shall say that $\ell_p$ is uniformly complementably finitely representable in $X$; we noted that this implies that $X$ contains a complemented copy of $L_p$.

\subsection{The hilbertian case}

We now concentrate  on the case of Hilbertian/euclidean copies. 
The exact value of the (best) constant of complementation of isometric copies of $\ell_2^n$ or $\ell_2$ inside $L_p$, $1 \leq p \leq +\infty$ is known (see Gordon-Lewis-Retherford \cite[Theorem 6]{Gor72}). Let us note:
\begin{equation}\label{o9895t45t5}
c_2(L_p)=\frac{2}{\sqrt{\pi}} \Big(\Gamma(\frac{p+1}{2})\Big)^{1/p} \Big(\Gamma(\frac{p'+1}{2})\Big)^{1/p'},\, 1<p<+\infty    
\end{equation}
and
$$c_2(L_\infty)=c_2(L_1)=+\infty$$
where for $p \in [1,+\infty]$ we use the notation $p' \in [1,+\infty]$ for the conjugate of $p$.

 Recall that $p(X):=\sup\conj{p}{ X \ {\rm has\ type\ } p}$ and
that $q(X):=\inf\conj{p}{ X \ {\rm has\ cotype\ } q}$. We say that $X$ is near Hilbert when $p(X)=q(X)=2$.  It is well-known that $X$ does  not necessarily have type $p(X)$ or cotype $q(X)$. However, the Maurey-Pisier Theorem \cite{MaPi} states that both $\ell_{p(X)}$ and $\ell_{q(X)}$ are finitely representable  in $X$.  Regarding duality, it is immediate that if $X$ has type $p>1$ then $X^*$ has cotype $p'$. A profound result  by Pisier \cite{Pi} states conversely that if $X$ has no-trivial type and cotype $q$, then $X^*$ has type $q'$. Another classsical result by  Pisier \cite{Pi} states that a Banach space $X$ has non-trivial type exactly when $X$ is  {\em locally $\pi$-euclidean}, that is, when  for some $c \geq 1$, for every $n$ and $\epsilon>0$, there exists $N$ such that any $N$-dimensional subspace of $X$ admits an $n$-dimensional subspace which is $1+\epsilon$-isometric to an euclidean space and which is $c$-complemented in $X$ (see also Pe\l czynski-Rosenthal \cite[Definition and Observation on p284]{PR}, noting that the definitions in \cite{Pi} and \cite{PR} are equivalent by Dvoretsky theorem).


\begin{lemma}\label{c2}  The following hold: 
\begin{enumerate}[\rm(1)]
\item   $c_2(L_p)=c_2(L_{p'})$ for $1 \leq p \leq +\infty$. 

\item
If $1\le p <\infty$ is in $\mathrm{Fr}(X)$ for some Fraïssé space $X$,  then
$c_2(L_p) \leq c_2(X).$

\item The map $p \mapsto c_2(L_p)$ is decreasing on $[1,2]$ and increasing on $[2,\infty]$.

\end{enumerate}
\end{lemma}

\begin{proof} Assertion (1) is a direct consequence of the formula \eqref{o9895t45t5}.  (2): Since $X$ is Fraïssé and $p\in \fr(X)$, it follows that $X$ contains an isometric copy of $L_p$. For a given $n$, if $\ell_2^n$ is $C$-complemented in $X$, then any copy of $\ell_2^n$ inside $L_p$ is also $C$-complemented in $X$ (because $X$ is Fraïssé) and consequently also in $L_p$. This implies that $c_2^n(L_p)\le c_2^n(X)$, and by Proposition \ref{complecomple} (3) we have that $c_2(L_p)\le c_2(X)$.
(3): since $L_q$ embeds isometrically into $L_p$ for $1 \leq p \leq q \leq 2$, this and assertion (2) implies that the map is decreasing on $[1,2]$, and then, by assertion (1), increasing on $[2,+\infty]$ by (2).   
\end{proof}

The next relates $p(X)$, $q(X)$, and complemented copies of   $L_p$-spaces.

\begin{proposition}\label{several} Let $X$ be a Fraïss\'e Banach space.
\begin{enumerate}[\rm(1)]
\item If $q(X)=+\infty$ then $X$ is contains an isometric copy of the Gurarij space, and if $X$ separable then $X$ is isometric to the Gurarij space.
\item  If $p(X)=1$ and $q(X)<+\infty$, then $X$ contains isometric copies of $L_1$, 
$\ell_2$ is uncomplemented in $X$, and $c_2^n(X) \geq n \cdot \Gamma(n/2)/(\sqrt{\pi}\cdot \Gamma(1/2+n/2))$.
\item If $p(X)>1$ then $\ell_2$ is complemented in $X$.
There exists a unique $1 < p  \leq 2$ such that 
$c_2(X)=c_2(L_p)$, and we have $p(X) \geq p$ and $q(X) \leq p'$.
\item If $p(X)>1$ and $X$ has type $p(X)$ (resp. has cotype $q(X)$) then $X$ contains a complemented isometric copy of $L_{p(X)}$ (resp. of $L_{q(X)}$).
\item  If $p(X)>1$ and $q(X) \in 2\N+4$, then $X$ does not have cotype $q(X)$.

\end{enumerate}
\end{proposition}

\begin{proof}
Assertion (1) was already observed in \cite{FLMT}. It is a consequence of the Maurey-Pisier Theorem in \cite{MaPi} that $\ell_\infty$ is finitely representable in any space with non non-trivial cotype, implying that a space satisfying (1) must be isometrically universal for separable spaces, and the fact that the Gurarij space is the unique separably universal Fra\"iss\'e space.

 Assertion (2) is a consequence of Maurey-Pisier Theorem and the Fraiss\'e property - Lemma \ref{c2} (2). Since $\ell_2$ is uncomplemented in $L_1$ it has to be uncomplemented in $X$ and this also provides a lower bound for $c_2^n(X)$ by  
 $c_2^n(L_1)=n \cdot \Gamma(n/2)/(\sqrt{\pi}\cdot \Gamma(1/2+n/2))$.

Regarding (3), since $X$ has non-trivial type, it is locally $\pi$-euclidean. Therefore $X$ contains uniformly complemented $1+\varepsilon$-copies of euclidean spaces of arbitrary size and therefore a complemented copy of the Hilbert space. The formula \eqref{o9895t45t5}  implies that $p\mapsto c_2(L_p)$ is continuous, and by Lemma \ref{c2} (3), is also decreasing on $[1,2]$. Since $c_2(L_1)=\infty$ and $c_2(L_2)=1$, it follows that there must be a unique   $1 < p  \leq 2$ such that 
$c_2(X)=c_2(L_p)$. By the Maurey-Pisier Theorem on the finite representability of $L_{p(X)}$ in $X$, and by using that $X$ is Fraïssé,   $X$ contains an isometric copy of $L_{p(X)}$. It follows that 
$c_2(L_{p(X)}) \leq c_2(X)=c_2(L_p)$, and therefore $p(X) \geq p$. The same reasoning holds for cotype.

(4): Following \cite[Proposition 13.16]{Mi-Sch}, if we suppose that $p(X)>1$ and $X$ has type $p(X)$,  then $X$ contains for some $C \geq 1$ and for any $\varepsilon>0$, $1+\varepsilon$-isomorphic copies of $\ell_p^n$ $C$-complemented, with $p=p(X)$, and therefore a $C$-complemented copy of $L_{p(X)}$ by Proposition \ref{complecomple} (4).

Now suppose that $p(X)>1$ and $X$ has cotype $q(X)$. Then the Pisier duality result \cite{Pi} holds to deduce that $p(X^*)=q(X)'$ and $X^*$ has type $p(X^*)$.
Therefore  as above $X^*$ contains $1+\varepsilon$-isomorphic copies of $\ell_{p(X^*)}^n$'s uniformly complemented, and by duality and local reflexivity, $X$ contains $1+\varepsilon$-isomorphic copies of $\ell_{q(X)}^n$'s uniformly complemented. The rest of the reasoning is as under the first hypothesis, to deduce that $X$ admits a complemented copy of $L_{q(X)}$. 

(5) is then a direct consequence of (4) and of Remark \ref{complu}.
\end{proof}

Using Proposition \ref{newproj}, we have the following striking corollary of Proposition \ref{several}:

\begin{corollary} 
Any Fra\"iss\'e space $X$ satisfies exactly one of the two following properties: either
\begin{enumerate}[\rm(1)]
    \item  $X$ contains a complemented isometric copy of $\ell_2$, or
    \item   $X$ contains an isometric copy of $L_1$.\qed
    \end{enumerate} 
\end{corollary}

Applying Proposition \ref{several} to the case of $1$-complemented copies of $L_p$
spaces we also get:

\begin{proposition} Let $X$ be a Fra\"iss\'e Banach space.
Then
\begin{enumerate}[\rm(1)]
    \item
there are at most $2$ values of $p$ (which must be conjugate) such that $X$ contains a $1$-complemented copy of $L_p$. 
\item
if $X$ contains a $1$-complemented copy of $L_p$ and $p \leq 2$ (resp. $p \geq 2$), then $p(X)=p$ (resp.  $q(X)=p'$). In particular if $X$ contains a $1$-complemented copy of $L_2$ then $X$ is near Hilbert. 
\end{enumerate}

\end{proposition}

\begin{proof}  Suppose that $X$ contains a $1$-complemented copy of $L_p$. Then, by Lemma \ref{c2} (1), (2), $c_2(L_{p'})=c_2(L_p)\le c_2(X)$. On  the other, the 1-complementation of a copy of $L_p$ implies that $c_2(X)\le c_2(L_p)$.  This  determines the set $\{p,p'\}$. (2): if $p \leq 2$ then it follows that $p(X) \geq p$, but also since $X$ contains a copy of $L_p$, that $p(X) \leq p$. Likewise if $p \geq 2$ then it follows that $q(X)=p'$.\end{proof}

We do not know if any Fraïssé contains a 1-complemented copy of some $L_p$ and if it contains at most  such 1-complemented $L_p$. 
There are other properties of Fra\"iss\'e $L_p$ spaces which may be seen as natural steps towards the conjecture that all separable Fra\"iss\'e spaces are either $L_p$-spaces or the Gurarij space. We just list two of them here: must any Fra\"iss\'e space $X$ attain its type
and cotype?
 if  $X$ is Fra\"iss\'e and reflexive, must ${\rm Fr}(X)$ be either an interval of the form $[p,2]$ or a singleton $\{q\}$ for $2 \leq q <+\infty$?
The reader will easily find other natural conjectures in this vein.

\

We conclude with a solution to  a  weak version of the Mazur rotation problem posed by G. Godefroy (see  \cite[Problem 5.20, and the paragraph before it]{felixvalentinbeata}).

\begin{theorem} A Fra\"iss\'e space admitting $C_\infty$ bump functions must be isomorphic to a Hilbert space. \end{theorem}

\begin{proof} Assume $X$ is as above and not isomorphic to a Hilbert space. Since a space with a 
 Fr\'echet-differentiable bump function is necessarily Asplund, $X$ cannot contain a copy of $L_1$, whereby by the Proposition \ref{several} (2), $X$ has non trivial type. 
According to a Theorem of Deville, see \cite[Chapter V.4]{DGZ} for an exposition, this implies $q(X)$ is an even number larger than $2$ and $X$ has cotype $q(X)$, and this contradicts Proposition \ref{several} (5).\end{proof}

A Banach space $X$ is Lipschitz transitive \cite{CSwheeling} when there exists $C \geq 1$ such that for any two normalized vectors $x,y$ of $X$, there is a surjective isometry $T$ on $X$ such that
$Tx=y$ and $\|T-Id\| \leq C\|y-x\|$. According to \cite[Lemma 2.6]{CSwheeling}, any Lipschitz transitive norm on a separable space is $C_\infty$. We immediately deduce:

\begin{proposition} A separable Fra\"iss\'e space which is Lipschitz transitive must be isomorphic to Hilbert space. \qed 
\end{proposition}






\bibliographystyle{alpha}
\bibliography{bibliography}

\end{document}